\documentclass{amsart}

\usepackage{lmodern}
\usepackage[T1]{fontenc}
\usepackage[utf8]{inputenc}
\usepackage[english]{babel}
\usepackage{csquotes}
\usepackage[unicode=true, psdextra]{hyperref}
\usepackage[alphabetic, msc-links]{amsrefs}
\usepackage{bm}
\usepackage{mathrsfs}
\usepackage{todonotes}
\usepackage{mathtools}
\usepackage{enumitem}
\usepackage{tikz-cd}
\usepackage{microtype}

\newtheorem{theorem}{Theorem}[section]
\newtheorem{lemma}[theorem]{Lemma}

\newtheorem{definition}[theorem]{Definition}

\newtheorem{proposition}[theorem]{Proposition}
\newtheorem{corollary}[theorem]{Corollary}

\theoremstyle{remark}

\numberwithin{equation}{section}

\newcommand{\eps}{\varepsilon}
\newcommand{\Per}{\text{Per}}
\newcommand{\per}{\text{per}}
\newcommand{\Z}{\mathbb{Z}}
\newcommand{\N}{\mathbb{N}}

\DeclareMathOperator{\supp}{\text{supp}}

\begin{document}

\title{On generic subsystems of a Smale space}


\author{\'Angela Flores-Concha}
\address{Department of Mathematics, The Pennsylvania State University.}
\email{aaf5457@psu.edu}
\thanks{The author was partially supported by the National Science Foundation grant DMS-2247553.}

\date{\today}


\begin{abstract}
    We study properties of generic compact invariant subsystems of a fixed Smale space in the Hausdorff topology. For a transitive Smale space, we show that the hyperspace of all subsystems contains an open and dense family of NMC subsystems, a dynamically degenerate class with finite orbit structure. A substantially different picture emerges in the closure of the infinite transitive locus: minimality, unique ergodicity, and zero topological entropy are generic, while subsystems conjugate to regular Toeplitz shifts are dense. Despite generic zero entropy, typical subsystems exhibit large oscillations in complexity: we prove that they have infinite slow entropy at the polynomial scale, and their slow entropy at the stretched-exponential scale is equal to one. For mixing Smale spaces, we further study the closure of the infinite totally transitive subsystems. In this hyperspace, minimality, unique ergodicity, and zero entropy remain generic; the conjugacy class of any Sturmian subshift is dense; and uniquely ergodic subsystems whose invariant measure is weakly mixing form a residual set. In particular, topological weak mixing is generic in this space.
\end{abstract}

\maketitle
    \section{Introduction}

        Smale spaces provide a topological framework for hyperbolic dynamics, encompassing shifts of finite type and the basic sets of Axiom A diffeomorphisms while retaining the fundamental features of uniform hyperbolicity, such as expansivity, shadowing, and local product structure. The purpose of this paper is to describe the generic behavior of compact invariant subsystems of a fixed Smale space in the sense of Baire category.

        The notion of Baire category, introduced at the end of the nineteenth century, has long been used to provide a topological description of typical behavior in several parameter spaces that naturally arise in the framework of topological dynamics. 
        In the early literature, Oxtoby and Ulam \cite{OxtobyUlam} proved that, for a polyhedron with a suitable Lebesgue--Stieltjes measure, ergodicity is generic among measure-preserving homeomorphisms in an appropriate uniform topology. In 1944, Halmos \cite{Halmos} considered the group of Lebesgue-preserving automorphisms of the unit interval and proved that the set of weakly mixing transformations is residual in the weak topology. This result was complemented by Rokhlin \cite{Rohlin}, who established that, in the same setting, strongly mixing automorphisms form a set of first category. A unifying theme in all these results is that the underlying measure space is fixed, while the measure-preserving transformation varies. Related questions were later studied in spaces of homeomorphisms; for instance, Kechris and Rosendal \cite{KechrisRosendal} proved that the homeomorphism group of the Cantor set has a residual conjugacy class.

        By the early 1960s, genericity questions in the space of invariant probability measures on a fixed dynamical system were also being considered. Parthasarathy, in \cite{ParthasarathyErgodic} and \cite{ParthasarathyMixing}, studied generic ergodic and mixing measures for shift transformations. Subsequently, Sigmund \cite{SigmundAxiomA} developed this point of view in the setting of differentiable hyperbolic dynamics. 

        A further possibility is to fix the ambient dynamics and study genericity within the space of invariant subsystems, endowed with the Hausdorff topology. Hochman \cite{Hochman} approached generic topological dynamics from two different perspectives, studying homeomorphisms of the Cantor set and a universal hyperspace of subshifts over the Hilbert cube, and establishing a correspondence between genericity in the two models. The hyperspace point of view was later reconsidered by Frisch and Tamuz \cite{FrischTamuz}, who studied generic entropy properties of spaces of finite-alphabet subshifts for finitely generated amenable group actions.

        Most directly related to the present work, Pavlov and Schmieding \cite{PavlovSchmieding} recently obtained a detailed description of generic one-dimensional subshifts. Their results show a dramatic contrast between the full space of subshifts and the closures of the infinite transitive and totally transitive loci: while generic behavior in the full space is highly degenerate, substantially richer dynamics becomes typical after imposing transitivity conditions.
        Their work provides the immediate symbolic precedent for the present paper. A subshift of finite type is itself a Smale space, and every topologically transitive Smale space admits a finite-to-one symbolic presentation by an irreducible shift of finite type. This makes the passage from subshifts to Smale spaces a natural next step. However, the extension is not formal, since the induced maps on hyperspaces do not automatically preserve residuality. The hyperbolic structure of the ambient Smale space therefore enters essentially into the analysis.

        In the space $S(X,f)$ of subsystems of a Smale space $(X,f)$, equipped with the Hausdorff topology, we show that a rather degenerate class of subsystems is generic. The typical properties are summarized in the following theorem. 
        \begin{theorem}\label{thm:Ssummary}
            Let $(X,f)$ be an irreducible Smale space. The typical subsystem in $S(X,f)$ is topologically conjugate to a countable sofic subshift which has at most linear complexity. Moreover, if it contains more than a single periodic orbit, then it is not topologically transitive.
        \end{theorem}

        \begin{theorem}\label{thm:Tsummary}
            Let $(X,f)$ be an irreducible Smale space. The typical subsystem in $\overline{T'(X,f)}$ is minimal, uniquely ergodic, and has zero topological entropy. However, it has infinite polynomial slow entropy, and stretched-exponential slow entropy equal to 1.

            Moreover, the family of subsystems conjugate to regular Toeplitz subshifts that factor onto the universal odometer is dense in $\overline{T'(X,f)}$.
        \end{theorem}

        We obtain the result about slow entropy using the word complexity of a suitable symbolic presentation of a typical subsystem.

        Lastly, restricting further to the space $TT'(X,f)$ of infinite and totally transitive subsystems, we obtain similar dynamical properties. Note that the results are established in the setting of a mixing Smale space to avoid degenerate subspaces.

         \begin{theorem}\label{thm:TTsummary}
            Let $(X,f)$ be a mixing Smale space. The typical subsystem in $\overline{TT'(X,f)}$ is minimal, uniquely ergodic, and has zero topological entropy. The unique invariant measure of such a subsystem is weakly mixing.

            Moreover, the conjugacy class of a Sturmian subshift is dense in $\overline{TT'(X,f)}$.
        \end{theorem}      

        The paper is organized into five sections. In Section \ref{sec:DefinitionSetting} we collect the definitions needed for the arguments presented throughout the article, and we precisely define the setting of this work. Section \ref{sec:AllSubsystems} treats the space of all subsystems $S(X,f)$, where we make essential use of shadowing and expansivity of the ambient space to prove Theorem \ref{thm:Ssummary}. Section \ref{section:toptransitive} deals with the space $\overline{T'(X,f)}$ of infinite transitive subsystems and focuses on the proof of Theorem \ref{thm:Tsummary}. Genericity of the topological properties is obtained by suitable open-set characterizations of these properties and by analyzing their behavior under the finite-to-one factor map from a subshift of finite type onto the ambient space. Complexity and slow entropy results require some symbolic tools that we develop there. In Section \ref{sec:totallytransitive} we focus on the space $\overline{TT'(X,f)}$ of infinite totally transitive subsystems, where we first establish results on weakly mixing measures and then use them to obtain the genericity results for subsystems stated in Theorem \ref{thm:TTsummary}.

        In each of the sections, the existence of a symbolic cover given by a Markov partition is essential. The induced map from the space of subshifts of this cover onto $S(X,f)$ allows us to transfer results proved in \cite{PavlovSchmieding} for the symbolic case to the Smale space setting. However, a $G_\delta$ characterization of a dynamical property is not preserved \emph{a priori}, since the symbolic cover is not necessarily injective. The same problem occurs with some of the properties we study, which require intrinsic arguments to obtain the desired results. 

        \section*{Acknowledgments}

        The author would like to thank Scott Schmieding for proposing this problem, and for the helpful discussions and comments that arose during the preparation of this work.

    \section{Definitions, setting and preliminaries}\label{sec:DefinitionSetting}
        \subsection{Definition of a Smale space and related properties}
        \begin{definition}[Cf.~\cite{Ruelle} or~\cite{PutnamMemoir}] A Smale space, also known as an $\mathscr{L}$-hyperbolic homeomorphism or a hyperbolic homeomorphism, is a dynamical system $(X,f)$ consisting of a homeomorphism $f\colon X \to X$ on a compact metric space $(X,d)$ with the following properties:

            There exist constants $\eps_X>0$, $\lambda\in(0,1)$ and a continuous map $[\cdot\,,\cdot]\colon \Delta_{\eps_X}\to X$ on $\Delta_{\eps_X}\coloneq \{(x,y)\in X\times X \mid d(x,y)<\eps_X\}$ that satisfies the sets of axioms of a Local Product Structure \textup{(LPS)}:
            \begin{enumerate}[label=(LPS\arabic*),ref=(LPS\arabic*), font=\upshape]
                \item \label{LPS1} $[x,x]=x$ for every $x\in X$;
                \item \label{LPS2} $[x,[y,z]]=[x,z]$ and $[[x,y],z]=[x,z]$ for $x,y,z\in X$ whenever both sides are defined;
                \item \label{LPS3} $[f(x),f(y)]=f([x,y])$ for $x,y\in X$ whenever both sides are defined;
            \end{enumerate}
             and the axioms of Uniform Contraction and Expansion \textup{(UCE)}:
            \begin{enumerate}[label=(UCE\arabic*),ref=UCE\arabic*, font=\upshape]
                \item[(UC)] $d(f^n(y),f^n(z))\leq \lambda^n d(y,z)$ for $n>0$ if $[y,x]=x=[z,x]$ for some $x\in X$;
                \item[(UE)] $d(f^{-n}(y),f^{-n}(z))\leq \lambda^n d(y,z)$ for $n>0$ if $[x,y]=x=[x,z]$ for some $x\in X$.
            \end{enumerate}
        \end{definition}

        For $0<\delta\leq \eps_X$ define the local stable and local unstable sets at $x$ of size $\delta$ by
            \[X^s(x,\delta)\coloneq \{y \in X \mid d(x,y)<\delta,\,[x,y]=y\},\]
            \[X^u(x,\delta)\coloneq \{y\in X \mid d(x,y)<\delta,\,[y,x]=y\}.\]
        The global stable and unstable sets at $x$ are defined as usual by 
        \[X^s(x)\coloneq \{y\in X \mid d(f^n(x),f^n(y))\to 0,\, n\in\mathbb{N}\},\] 
        \[X^u(x)\coloneq \{y\in X \mid d(f^{-n}(x),f^{-n}(y))\to 0,\, n\in\mathbb{N}\}.\]
        The set of axioms \textup{(LPS)} implies, for each $x\in X$, the existence of a homeomorphism from $X^s(x,\delta)\times X^u(x,\delta)$ to a neighborhood of $x$ for $\delta$ independent of $x$ and small enough. Hence we can locally describe points in terms of stable and unstable coordinates. 
        Additionally, the set of axioms \textup{(UCE)} guarantees uniform contraction along stable sets and uniform expansion along unstable sets.

        The definition presented here is a special case of a type of system known as topologically Anosov, a topic treated extensively in \cite{AokiHiraide}. Under the more general assumptions of a topologically Anosov system, it can be shown that Smale spaces exhibit the expansivity property and the shadowing property (or pseudo-orbit tracing property POTP). Since both are frequently used in our proofs, we recall their definitions below. 
         
        \begin{definition}[Expansivity]
            A topological dynamical system $(X,f)$ on a compact metric space $(X,d)$ is expansive if there exists $c_f>0$ such that
            \[d(f^n(x),f^n(y))\leq c_f \text{ for every }n\in\mathbb{Z} \Longrightarrow x=y.\]  
        \end{definition}

        \begin{definition}[Shadowing Property]
            Let $S$ be a finite or infinite interval of $\mathbb{Z}$. Given $\eps>0$ there exists $\alpha>0$ so that every $\alpha$-pseudo-orbit $\{x_k\}_{k\in S}$ in $X$ is $\eps$-shadowed by a point $x\in X$, i.e.,
                \[d(f(x_k),x_{k+1})<\alpha \,\forall k, k+1\in S \Longrightarrow \exists x\in X \text{ so that } d(f^k(x),x_k)<\eps \,\forall k\in S.\]
        \end{definition}

        These two essential properties of a Smale space allow the construction of the primary tool used to prove the main results, namely, the existence of Markov partitions of arbitrarily small diameter, and of subshifts of finite type (SFT) that code the Smale space according to the selected Markov partition. For the definition of a Markov partition and a proof of the construction of such partitions, see \cite{AokiHiraide} for topologically Anosov maps or \cite{Bowen} for the differentiable setting.

        We need the following definition for the complexity results in Section \ref{section:toptransitive}.
        \begin{definition}\label{def:suresolving}
            Let $(X,f)$ and $(Y,g)$ be Smale spaces and let $\pi:X\to Y$ be a map that intertwines the dynamics of $(X,f)$ and $(Y,g)$, i.e., $\pi\circ f=g\circ \pi$. We say that $\pi$ is $s$-resolving (or $u$-resolving) if, for any $x\in X$, its restriction to $X^s(x)$ (or $X^u(x)$, respectively) is injective.
        \end{definition}

        \subsection{Definitions in topological dynamics and ergodic theory}
        We use some standard terminology from topological dynamics and ergodic theory throughout the article. Although it is probably familiar to the reader, we recall the relevant concepts here for completeness.
        
        First, we define irreducibility, a term we use interchangeably with topological transitivity throughout the text. 

        \begin{definition}\label{def:toptrans}
            A topological dynamical system $(X,f)$ is topologically transitive (or irreducible) if for every pair of nonempty open sets $U$ and $V$ there exists $n>0$ such that $f^n(U)\cap V\neq \emptyset$.
        \end{definition}

        The system $(X,f)$ is said to be point transitive if there exists a transitive point. That is, there exists $x\in X$ such that $\{f^n(x)\}_{n\in\N}$ is dense in $X$.
        This definition of point transitivity differs from the definition of topological transitivity given above. However, as proved in \cite{TopTrans}, both definitions are equivalent whenever $X$ is perfect and non-meager. In Section \ref{section:toptransitive}, we consider subsystems satisfying these properties, thus we use both definitions interchangeably.

        In our setting, we will mainly consider irreducible systems. However, we require the stronger assumption of topological mixing in Section \ref{sec:totallytransitive}.
        \begin{definition}\label{def:topmixing}
             A topological dynamical system $(X,f)$ is topologically mixing if for every pair of open sets $U,V\neq\emptyset$, there exists $N\in\N$ such that $f^n(U)\cap V\neq\emptyset$ for every $n\geq N$.
        \end{definition}

        A topologically transitive Smale space $(X,f)$ admits a spectral decomposition: $X$ is the union of $p$ disjoint compact sets $\{X_k\}_{k=1}^p$, which are cyclically permuted by $f$, i.e., $f(X_k)=X_{k+1}$, and such that $f^p\left|_{X_k}\right.$ is topologically mixing. A proof of this fact in the differentiable setting can be found in \cite{Bowen}.
        If $(X,f)$ is topologically mixing, the spectral decomposition is trivial and $p=1$. This is the connection between transitivity and mixing in the Smale space setting.

        There is a third important definition that lies between transitivity and mixing, which is central in Section \ref{sec:totallytransitive}: total transitivity. 
        \begin{definition}
            Let $(X,f)$ be a dynamical system. We say that $(X,f)$ is totally transitive if $(X,f^n)$ is transitive for every $n\in\N$.
        \end{definition}

        The orbit of $x\in X$ is defined by $O(x)\coloneq\{f^n(x) \mid n\in \Z\}$, its forward orbit is $O^+(x)\coloneq\{f^n(x) \mid n\in\N_0\}$ and its backward orbit is $O^-(x)\coloneq \{f^{-n}(x): n\in\N\}$. The $\alpha$-limit of $x\in X$ is defined by $\alpha(x)\coloneq \bigcap_{n\in\N}\overline{O^-(f^{-n}(x))}$ and its $\omega$-limit by $\omega(x)\coloneq\bigcap_{n\in\N}\overline{O^+(f^n(x))}$. The set of periodic points of $(X,f)$ is denoted by $\Per(X)$ if the map $f$ is implied. A point that is not periodic is called aperiodic.

        A factor map between topological dynamical systems $\pi\colon(X,f)\to (Y,g)$ is a continuous and surjective map that intertwines the dynamics of $X$ and $Y$, i.e., $\pi\circ f=g\circ\pi$. If such a $\pi$ exists, we say that $(Y,g)$ is a factor of $(X,f)$, or equivalently, that $(X,f)$ is an extension of $(Y,g)$. A factor map $\pi\colon X\to Y$ is a topological conjugacy if $\pi$ is also a homeomorphism. If a conjugacy between $(X,f)$ and $(Y,g)$ exists, we say that $(X,f)$ and $(Y,g)$ are topologically conjugate. 

        A factor map $\pi\colon X\to Y$ is finite-to-one if $|\pi^{-1}(y)|<\infty$ for every $y\in Y$. In our setting, finite-to-one maps are bounded-to-one, i.e., $\sup_{y\in Y}|\pi^{-1}(y)|<\infty$, and we use both notions indistinguishably. The map $\pi$ is called almost one-to-one if the set $\{y\in Y \mid |\pi^{-1}(y)|=1\}$ is a $G_\delta$ dense subset of $Y$.

        Given a dynamical system $(X,f)$ on a compact metric space $(X,d)$, the $n$-th Bowen metric $d_n$ on $X$ is defined by $d_n(x,y)\coloneq\max_{0\leq j<n}d(f^j(x),f^j(y))$.

        For $n\in\N$ and $r>0$ fixed, we say that a set $E\subseteq X$ is an $(n,r)$-separated set of $X$ if $d_n(x,y)>r$ for every distinct pair $x,y\in E$. If $X$ is compact, the maximal cardinality of an $(n,r)$-separated set is well defined and denoted by $S_X(n,r)$.

        The topological entropy of a dynamical system $(X,f)$ is defined by 
            \[h(f,X)=\lim_{\eps\to 0^+}\limsup_{n\to\infty}\dfrac{\log S_X(n,\eps)}{n}.\] 
        If $(X,f)$ is expansive, there exists $\eps\in(0,c_f)$ such that 
            \[h(f,X)=\limsup_{n\to\infty}\dfrac{\log S_X(n,\eps)}{n}.\]

        A measurable dynamical system $(X,\mathcal{B},f,\mu)$ is a compact metric space $X$ together with its Borel $\sigma$-algebra $\mathcal{B}$, a measurable function $f\colon X\to X$ and an $f$-invariant Borel probability measure $\mu$. For a triple $(X,\mathcal{B},f)$ as before, we define $\mathcal{M}_f(X)$ as the space of $f$-invariant Borel probability measures on $(X,\mathcal{B})$. This space is compact and metrizable with the weak-\textasteriskcentered\ topology.
        \begin{definition}
            A measurable dynamical system $(X,\mathcal{B},f,\mu)$ is weakly mixing if, for any measurable sets $A,B\in\mathcal{B}$, 
                \[\lim_{n\to\infty}\dfrac{1}{n}\sum_{j=0}^{n-1}\left|\mu(f^{-j}(A)\cap B)-\mu(A)\mu(B)\right|=0.\]
        \end{definition}

        \begin{definition}
            Let $(X,\mathcal{B},f,\mu)$ be a measurable dynamical system. The measure $\mu$ is called periodic if it is supported on a periodic orbit.
        \end{definition}

        \subsection{Definitions in symbolic dynamics}

        Let $m\geq 2$ be a fixed integer and consider the finite set $\mathcal{A}=\{0,1,\ldots,m-1\}\subseteq \N_0$, which we call the alphabet. Consider the set of sequences \(\Sigma_m\coloneq\{x=(x_n)_{n\in\Z} \mid x_n\in\mathcal{A}\;\forall n\in \Z\}\) with the metric given by
                    \[d(x,y)=2^{-\max\{i \in \N \mid x_j=y_j \forall |j|\leq i\}}.\]
        This metric induces the product topology on $\mathcal{A}^\Z$, which implies that $(\Sigma_m,d)$ is a compact metric space. 
        We define the left-shift map $\sigma\colon\Sigma_m\to\Sigma_m$, or simply the shift map, by $(\sigma(x))_{n}=x_{n+1}$, which is a homeomorphism. The invertible topological dynamical system $(\Sigma_m,\sigma)$ is called the full shift in $m$ symbols. 

        A subshift $(\Sigma,\sigma)$ is a subsystem of $(\Sigma_m,\sigma)$. Namely, $\Sigma$ is a closed subset of $\Sigma_m$ and it is $\sigma$-invariant, i.e., $\sigma(\Sigma)=\Sigma$.

        Denote by $\mathcal{A}^n$ the set of $n$-strings with symbols in $\mathcal{A}$, and by $\mathcal{A}^*$ the free monoid of all finite strings, i.e., $\mathcal{A}^*\coloneq\bigsqcup_{n\geq 0}\mathcal{A}^n$. A word $w$ is an element of $\mathcal{A}^*$, and its length, denoted by $|w|$, is $n$ if $w\in\mathcal{A}^n$.
        Given a word $w$, we can define a cylinder set associated to it by \(C_w=\{x\in\Sigma_m \mid  x_{[0,|w|)}=w\},\) where $x_{[0,n)}$ denotes the $n$-string $x_0x_1\ldots x_{n-1}$. The cylinder sets and their images and preimages under $\sigma$ form a basis for the topology.

        Given a subshift $(\Sigma,\sigma)$, we say that a word $w$ is $\Sigma$-admissible if $C_w\cap\Sigma\neq\emptyset$. A word $w$ is $\Sigma$-forbidden if it is not $\Sigma$-admissible. The language of $\Sigma$, denoted by $\mathcal{L}(\Sigma)$, is the set of all $\Sigma$-admissible words, i.e.,
            \[\mathcal{L}(\Sigma)=\left\{w\in\mathcal{A}^* \mid \exists x\in \Sigma \text{ such that } x_{[0,|w|)}=w\right\}.\] 
        The $n$-th language of $\Sigma$, denoted by $\mathcal{L}_n(\Sigma)$, is the subset of all $\Sigma$-admissible words of length $n$, i.e., $\mathcal{L}_n(\Sigma)=\mathcal{L}(\Sigma)\cap \mathcal{A}^n$.

        \begin{definition}[Word complexity]
         Let $(\Sigma,\sigma)$ be a subshift. Its word complexity function is defined by
            \[C(\Sigma,\cdot)\colon \N\to\N,\quad C(\Sigma,n)\coloneq |\mathcal{L}_n(\Sigma)|.\]
        \end{definition}
        
        Consider a set of words $\mathcal{F}\subseteq \mathcal{A}^*$. We can define a subshift $\Sigma_\mathcal{F}$ as the subset of $\Sigma_m$ whose set of forbidden words is exactly $\mathcal{F}$. Every subshift can be obtained in this way with a countable $\mathcal{F}$. We say that $\Sigma$ is a subshift of finite type (SFT) if there exists a finite set of words $\mathcal{F}\subseteq \mathcal{A}^*$ such that $\Sigma=\Sigma_\mathcal{F}$.

        The $n$-th Rauzy graph of a subshift $\Sigma$ is the directed graph $\Gamma(\Sigma,n)$ with vertex set $\mathcal{L}_{n-1}(\Sigma)$, and directed edges from $w_1\cdots w_{n-1}$ to $w_2\cdots w_n$ for every $w_1\ldots w_n\in \mathcal{L}_n(\Sigma)$.

        The $n$-step SFT approximation of a subshift $\Sigma$, denoted by $\Sigma^{(n)}$, is the edge shift generated by the Rauzy graph $\Gamma(\Sigma,n)$. In other words, it is the space of all bi-infinite walks in $\Gamma(\Sigma,n)$. In the Hausdorff metric defined in Subsection \ref{subsec:Preliminaries}, $d_\mathcal{H}\left(\Sigma,\Sigma^{(n)}\right)<2^{-n}$.

        Now we consider two subshifts $X$ and $Y$ over alphabets $\mathcal{A}$ and $\mathcal{B}$ respectively. Let $a, m\in\Z$ be such that $-m\leq a$ and $\phi\colon\mathcal{L}_{m+a+1}(X)\to \mathcal{B}$ be a block map, i.e., a function, also known as local rule, that maps $X$-admissible words of length $m+a+1$ to a symbol in the alphabet $\mathcal{B}$. The map $\varphi\colon X\to \mathcal{B}^{\Z}$ defined by $y=\varphi(x)$, where
            \[y_i=\phi(x_{i-m}x_{i-m+1}\cdots x_{i+a})=\phi\left(x_{[i-m,i+a]}\right),\]
        is called a sliding block code with memory $m$ and anticipation $a$ induced by $\phi$. If $\varphi(X)\subseteq Y$, we write $\varphi\colon X\to Y$.

        This type of map between subshifts satisfies $\varphi\circ \sigma_X=\sigma_Y\circ\varphi$, where $\sigma_X$ and $\sigma_Y$ denote the shift maps in $X$ and $Y$ respectively. Therefore, if $\varphi$ is also surjective, a sliding block code $\varphi\colon X\to Y$ is a factor map. By the Curtis--Lyndon--Hedlund theorem, sliding block codes describe all possible factor maps between subshifts.

        \subsection{Space of subsystems and setting}\label{subsec:Preliminaries}
        Let $(X,f)$ be a topological dynamical system with $X$ a compact metric space and $f\colon X\to X$ a continuous function. Define the space 
            \[S(X,f)\coloneq \{Y\subseteq X \mid Y \text{ is compact and } f(Y)=Y\}\]
        of all subsystems of $(X,f)$.
        We endow this space with the Hausdorff metric, given by
            \[d_\mathcal{H}(Y,Y')\coloneq \max\left\{\sup_{y\in Y}d(y,Y'),\sup_{y'\in Y'}d(y',Y)\right\}\]
        for $Y,Y'$ compact $f$-invariant subsets of $X$. The compactness of the metric space $X$ implies that $(S(X,f),d_\mathcal{H})$ is compact as well. 
        
        Define the Hausdorff ball centered at $Y$ of radius $r$ by \[B_\mathcal{H}(Y,r)\coloneq\{Y'\in S(X,f) \mid d_\mathcal{H}(Y,Y')<r\}.\]

        If $\pi\colon (X,f)\to (Y,g)$ is a factor map between topological dynamical systems with $X,Y$ compact metric spaces, then $\pi$ induces a map $\pi_*\colon S(X,f)\to S(Y,g)$ defined by $Z\mapsto \pi(Z)$.

        \begin{proposition}\label{prop:contsurjpi}
            The map $\pi_*\colon S(X,f)\to S(Y,g)$, $Z\mapsto \pi(Z)$, induced by a topological semiconjugacy $\pi\colon (X,f)\to (Y,g)$ between compact metric spaces, is well-defined, continuous in the Hausdorff metric, and surjective.
        \end{proposition}

        \begin{proof}
            We first prove that the map $\pi_*$ is well defined. If $Z\in S(X,f)$, that is, $Z$ is compact and $f(Z)=Z$, then $\pi(Z)$ is a compact subset of $Y$ by continuity of $\pi$. The semiconjugacy property gives $g(\pi(Z))=\pi(f(Z))=\pi(Z)$, so $\pi(Z)$ is $g$-invariant. Therefore, $\pi_*(Z)=\pi(Z)\in S(Y,g)$.

            In a similar fashion, we see that if $A\in S(Y,g)$, then $\pi^{-1}(A)\in S(X,f)$: $A$ is compact so $\pi^{-1}(A)$ is a closed subset of $X$ and hence compact. In addition, the semiconjugacy gives $f(\pi^{-1}(A))=\pi^{-1}(g(A))=\pi^{-1}(A)$. 
            This proves, in particular, that $\pi_*$ is surjective.

            Now we prove that $\pi_*$ is continuous: let $\eps>0$ and $\delta>0$ be such that $d_X(x,y)<\delta$ implies $d_Y(\pi(x),\pi(y))<\eps/2$, where $d_X$ and $d_Y$ denote the metrics in $X$ and $Y$ respectively. 
            Consider two arbitrary sets $A,B\in S(X,f)$ such that $d_\mathcal{H}(A,B)<\delta$. In particular, 
            \[\sup_{a\in A}d_X(a,B)<\delta, \qquad \sup_{b\in B}d_X(b,A)<\delta.\] 
            
            For any $a\in A$, the compactness of $B$ allows us to choose $b=b(a)\in B$ such that $d_X(a,B)=d_X(a,b(a))<\delta$. Thus, we obtain $d_Y(\pi(a),\pi(B))<\eps/2$. Since $a$ was arbitrary, we have
            \[\sup_{w\in \pi(A)}d_Y(w,\pi(B))=\sup_{a\in A}d_Y(\pi(a),\pi(B))\leq\dfrac{\eps}{2}<\eps.\] 
            
            Analogously, $\sup_{w\in\pi(B)}d_Y(w,\pi(A))<\eps$, which implies $d_\mathcal{H}(\pi_*(A),\pi_*(B))<\eps$. This concludes the proof.
        \end{proof}
        
        From now on we fix a topologically transitive Smale space $(X,f)$, a Markov partition of it with small diameter, and the corresponding SFT, which we denote by $(\Sigma,\sigma)$ or just $\Sigma$ if the left-shift $\sigma$ is implied. Then there exists a factor map $\pi\colon\Sigma\to X$ which is finite-to-one, almost one-to-one, and intertwines the dynamics of $(X,f)$ and $(\Sigma,\sigma)$ by $\pi\circ\sigma=f\circ\pi$. Both spaces $S(X,f)$ and $S(\Sigma,\sigma)$ are compact with these metrics. 
         
        The Hausdorff metric in $S(\Sigma,\sigma)$ can be explicitly written in terms of sizes of languages, as in \cite{PavlovSchmieding}: for $X, Y$ subshifts, $d_\mathcal{H}(X,Y)=2^{-\inf\{n \mid \mathcal{L}_n(X)\neq \mathcal{L}_n(Y)\}}.$

        From now on, $(Y,g)$ will always denote a compact metric space $Y$, which is usually a subset of $X$ or $\Sigma$, together with a homeomorphism $g\colon Y\to Y$ that is usually the restriction of the hyperbolic homeomorphism $f$ to $Y$ or the restriction of the left-shift $\sigma$. Note that $Y\in S(X,f)$ does not, in general, inherit the structure o a  Smale space.

        By Proposition \ref{prop:contsurjpi}, the factor map $\pi$ induces a map $\pi_*\colon S(\Sigma,\sigma)\to S(X,f)$ defined by $Z\mapsto \pi(Z)$. We refer to this map as the induced map from now on.

        To study typical subsystems in $S(X,f)$ and other relevant subspaces of it, we first introduce the notion of genericity. Since the framework is topological, a property $P$ is called generic (or typical) in a space $S$ if the set of subsystems satisfying $P$ is residual, i.e., if it contains a dense $G_\delta$ subset of $S$ with respect to the Hausdorff metric topology.
        
        Several results on the genericity of subshifts over finite alphabets from \cite{PavlovSchmieding} are used throughout this article. The proof strategy in almost every main result consists of two steps:
        First, given a generic property $P$ in $S(\Sigma,\sigma)$ or a related subspace, we analyze the image under $\pi_*$ of a subsystem satisfying $P$. The resulting property is typically weaker than $P$ and, by Proposition \ref{prop:contsurjpi}, it holds on a dense subset of $S(X,f)$.
        Second, we characterize this property in terms of $G_\delta$ sets. Unlike the first step, the $G_\delta$ requirement does not follow automatically unless $\pi_*$ is open, which is not true in general. Hence, we employ alternative formulations of the properties and problem-specific arguments to establish their residuality.

    \section{The space of all subsystems \texorpdfstring{$S(X,f)$}{S(X,f)}}\label{sec:AllSubsystems}

        The first family of subsystems we consider is that of NMC subsystems. The definition is based on the analogous dynamical property of No Middle Cycles (NMC) for subshifts introduced in \cite{PavlovSchmieding}. This property is defined in terms of the structure of their Rauzy graphs, which graphically represent valid transitions of words of a fixed length. In general, an irreducible component of such a graph has the form of a barbell, possibly with additional weight along the bar between its two ends, as in Figure \ref{fig:barbell}. 

        \begin{figure}[htbp]
        \centering
        \begin{minipage}{0.45\textwidth}
            \centering
                        \begin{tikzpicture}[
                x=0.75pt,
                y=0.75pt,
                yscale=-1,
                xscale=1,
                line width=0.75pt,
                >=stealth,
                scale=0.3
            ]


            \coordinate (L1) at (100.62,90.93);
            \coordinate (L2) at (189.28,66.86);
            \coordinate (L3) at (213.36,155.52);
            \coordinate (L4) at (124.89,179.21);

            \coordinate (P0) at (213.36,155.52);
            \coordinate (P1) at (289.38,195.14);
            \coordinate (P2) at (365.39,234.75);
            \coordinate (P3) at (441.41,274.37);
            \coordinate (P4) at (517.42,313.98);

            \coordinate (R1) at (517.42,313.98);
            \coordinate (R2) at (601.86,306.62);
            \coordinate (R3) at (577.36,384.49);


            \draw[->]
                (107.07,83.17)
                .. controls (124.87,59.43) and (156.64,50.77) ..
                (183.81,62.26);

            \draw[->]
                (195.79,71.12)
                .. controls (219.31,89.46) and (228.51,120.28) ..
                (218.30,146.68);

            \draw[->]
                (208.52,162.98)
                .. controls (190.79,185.03) and (161.26,193.56) ..
                (135.39,184.16);

            \draw[->]
                (118.32,174.52)
                .. controls (95.06,156.28) and (86.40,125.14) ..
                (97.27,98.48);


            \draw[->, shorten <=2pt, shorten >=2pt]
                (P0) -- (P1);

            \draw[->, shorten <=2pt, shorten >=2pt]
                (P1) arc[start angle=207.525,
                        end angle=387.525,
                        radius=42.856];

            \draw[->, shorten <=2pt, shorten >=2pt]
                (P2) arc[start angle=27.525,
                        end angle=207.525,
                        radius=42.856];

            \draw[->, shorten <=2pt, shorten >=2pt]
                (P2) -- (P3);

            \draw[->, shorten <=2pt, shorten >=2pt]
                (P3) -- (P4);


            \draw[->]
                (521.46,307.57)
                .. controls (528.60,297.67)
                            and (539.36,290.35) ..
                (552.28,287.79)
                .. controls (567.91,284.69)
                            and (583.28,289.24) ..
                (594.51,298.85);

            \draw[->]
                (606.48,313.98)
                .. controls (610.64,322.06)
                            and (612.65,331.39) ..
                (611.82,341.12)
                .. controls (610.32,358.51)
                            and (600.09,373.05) ..
                (585.83,380.80);

            \draw[->]
                (571.34,385.96)
                .. controls (561.78,387.78)
                            and (551.60,386.85) ..
                (542.01,382.67)
                .. controls (518.50,372.41)
                            and (506.85,346.26) ..
                (514.15,322.31);


            \fill (L1) circle[radius=3pt];
            \fill (L2) circle[radius=3pt];
            \fill (L3) circle[radius=3pt];
            \fill (L4) circle[radius=3pt];

            \fill (P1) circle[radius=3pt];
            \fill (P2) circle[radius=3pt];
            \fill (P3) circle[radius=3pt];

            \fill (R1) circle[radius=3pt];
            \fill (R2) circle[radius=3pt];
            \fill (R3) circle[radius=3pt];

            \end{tikzpicture}
            \caption{Barbell}\label{fig:barbell}
        \end{minipage}
        \begin{minipage}{0.45\textwidth}
            \centering
                \begin{tikzpicture}[
                        x=0.75pt,
                        y=0.75pt,
                        yscale=-1,
                        xscale=1,
                        line width=0.75pt,
                        >=stealth,
                        scale=0.3
                    ]


                    \coordinate (L1) at (100.62,90.93);
                    \coordinate (L2) at (189.28,66.86);
                    \coordinate (L3) at (213.36,155.52);
                    \coordinate (L4) at (124.89,179.21);

                    \coordinate (P0) at (213.36,155.52);
                    \coordinate (P1) at (289.38,195.14);
                    \coordinate (P2) at (365.39,234.75);
                    \coordinate (P3) at (441.41,274.37);
                    \coordinate (P4) at (517.42,313.98);

                    \coordinate (R1) at (517.42,313.98);
                    \coordinate (R2) at (601.86,306.62);
                    \coordinate (R3) at (577.36,384.49);


                    \draw[->]
                        (107.07,83.17)
                        .. controls (124.87,59.43) and (156.64,50.77) ..
                        (183.81,62.26);

                    \draw[->]
                        (195.79,71.12)
                        .. controls (219.31,89.46) and (228.51,120.28) ..
                        (218.30,146.68);

                    \draw[->]
                        (208.52,162.98)
                        .. controls (190.79,185.03) and (161.26,193.56) ..
                        (135.39,184.16);

                    \draw[->]
                        (118.32,174.52)
                        .. controls (95.06,156.28) and (86.40,125.14) ..
                        (97.27,98.48);


                    \draw[->, shorten <=2pt, shorten >=2pt] (P0) -- (P1);
                    \draw[->, shorten <=2pt, shorten >=2pt] (P1) -- (P2);
                    \draw[->, shorten <=2pt, shorten >=2pt] (P2) -- (P3);
                    \draw[->, shorten <=2pt, shorten >=2pt] (P3) -- (P4);


                    \draw[->]
                        (521.46,307.57)
                        .. controls (528.60,297.67)
                                    and (539.36,290.35) ..
                        (552.28,287.79)
                        .. controls (567.91,284.69)
                                    and (583.28,289.24) ..
                        (594.51,298.85);

                    \draw[->]
                        (606.48,313.98)
                        .. controls (610.64,322.06)
                                    and (612.65,331.39) ..
                        (611.82,341.12)
                        .. controls (610.32,358.51)
                                    and (600.09,373.05) ..
                        (585.83,380.80);

                    \draw[->]
                        (571.34,385.96)
                        .. controls (561.78,387.78)
                                    and (551.60,386.85) ..
                        (542.01,382.67)
                        .. controls (518.50,372.41)
                                    and (506.85,346.26) ..
                        (514.15,322.31);


                    \fill (L1) circle[radius=3pt];
                    \fill (L2) circle[radius=3pt];
                    \fill (L3) circle[radius=3pt];
                    \fill (L4) circle[radius=3pt];

                    \fill (P1) circle[radius=3pt];
                    \fill (P2) circle[radius=3pt];
                    \fill (P3) circle[radius=3pt];

                    \fill (R1) circle[radius=3pt];
                    \fill (R2) circle[radius=3pt];
                    \fill (R3) circle[radius=3pt];

            \end{tikzpicture}
            \caption{Simple barbell}\label{fig:simplebarbell}
        \end{minipage}
        \end{figure}

        The weights correspond to strongly connected components of the graph, and the bars connecting them correspond to its transient components. The NMC property requires every irreducible component of the Rauzy graph of a sufficiently large index to present one of the following structures: it must be a simple barbell, i.e., it contains weights only at its two ends, which must be simple cycles, connected by a simple directed path, as in Figure \ref{fig:simplebarbell} (our terminology differs from that of \cite{PavlovSchmieding}, which uses empty barbell instead of simple barbell); or an isolated simple cycle. Therefore, the condensation graph of the Rauzy graph only has sinks or sources, or equivalently, the only possible strongly connected components are sinks or sources. 
        Dynamically, this means that every aperiodic orbit represented by a simple barbell is asymptotic in forward and backward time to the periodic orbits associated with its two terminal cycles. Thus, every aperiodic orbit is biasymptotic to two different periodic orbits.

        The following definition of an NMC subsystem generalizes this property to the broader class of topological dynamical systems on metric spaces. Despite the absence of Rauzy graphs in the general setting, we preserve the terminology.

        \begin{definition}\label{def:NMC}
            A topological dynamical system $(Y,g)$ is called an NMC system if the following conditions hold: 
                \begin{enumerate}[font=\upshape]
                    \item There exist finitely many orbits in $Y$,
                    \item every aperiodic point is biasymptotic to two different periodic orbits in $Y$, and
                    \item $\alpha(y)\neq \omega(y')$ for every pair of aperiodic points $y,y'\in Y$.
                \end{enumerate} 
        \end{definition}

        Because the definition of an NMC subsystem is given in terms of dynamical objects that are preserved under topological conjugacy, a system that is topologically conjugate to an NMC subsystem is also NMC.

        The goal of this section is to prove the following result.

        \begin{theorem}\label{thm:NMC}
            Let $(X,f)$ be an irreducible Smale space. The family of NMC subsystems is discrete and dense in $S(X,f)$.
        \end{theorem}
        
        The proof of Theorem \ref{thm:NMC} is divided into several steps. Recall $\pi\colon(\Sigma,\sigma)\to (X,f)$ from the previous section. First, we describe the type of system obtained when applying $\pi_*$ to a subshift in $S(\Sigma,\sigma)$ with the NMC property. 
        
        A point $\pi(x)\in X$ is periodic if and only if $x\in \Sigma$ is periodic, because $\pi$ is finite-to-one. Thus, if $Z\in S(\Sigma,\sigma)$ is an NMC subshift, then $\pi_*(Z)$ contains finitely many periodic orbits. In addition, every aperiodic point in $\pi_*(Z)$ is biasymptotic to two periodic orbits by the continuity of $\pi$. However, both backward and forward limiting orbits need not be distinct. For instance, the coding of the cat map in \cite{BrinStuck} maps the closure of the heteroclinic orbit of $x=1^\infty.3^\infty$ to the closure of the homoclinic orbit of $q=\left(\frac{2}{\sqrt{5}},\frac{3\sqrt{5}-5}{10}\right)\in\mathbb{T}^2$.

        This implies, a priori, that the image under $\pi_*$ of an NMC subshift in $S(\Sigma,\sigma)$ may fail to satisfy the second condition in Definition \ref{def:NMC}.
        It is important to note that an NMC subshift has finitely many orbits, which is a direct consequence of Lemma 3.5 in \cite{PavlovSchmieding}, and so the image under $\pi_*$ of such a system must have finitely many orbits as well. 

        Density of the systems that are images of NMC subshifts is an immediate consequence of Proposition \ref{prop:contsurjpi} and Theorem 3.6 in \cite{PavlovSchmieding}.  If such systems can be approximated by NMC subsystems, this establishes the first part of the theorem. The approximation argument we use relies heavily on the shadowing property of the ambient space $X$.

        \begin{lemma}\label{lem:NMCdense}
            Let $(X,f)$ be an irreducible Smale space and $Y\in S(X,f)$ the image under $\pi_*$ of a subshift in $S(\Sigma,\sigma)$ with the NMC property. Then, for every sufficiently small $\eps\in (0,c_f)$ there exists an NMC subsystem $Y''\in S(X,f)$ such that $d_{\mathcal{H}}(Y,Y'')<\eps$.
        \end{lemma}

        \begin{proof}
            We have already established some properties that $Y$ inherits as the image under $\pi_*$ of a subshift with the NMC property, such as biasymptoticity to periodic orbits, a finite number of orbits, and finitely many periodic orbits in particular. From now on, we name orbits according to their limit behavior: periodic orbits; orbits biasymptotic to a single periodic orbit, which we call homoclinic; and orbits biasymptotic to two different periodic orbits, which we call heteroclinic.

            The proof is divided into two steps. First, we construct $Y'\in B_{\mathcal{H}}(Y,\eps/2)\cap S(X,f)$ with no homoclinic orbits. Second, we construct an NMC subsystem $Y''\in B_{\mathcal{H}}(Y',\eps/2)\cap S(X,f)$.
            To achieve the first goal, we replace the homoclinic orbits in $Y$ by nearby periodic orbits, as described below.

            Consider any of the finitely many homoclinic orbits in $Y$, say $O(y)$, and a periodic point $p\in Y$ such that 
                \[\lim_{n\to\infty} d(f^n(p),f^n(y))=0.\] 
            There exists $s\in \{0,\ldots,\per(p)-1\}$ so that 
                \[\lim_{n\to-\infty}d(f^n(y),f^n(f^s(p)))=0.\] 
            Let $\alpha\in (0,\eps/2)$ be given by the shadowing property corresponding to $\eps/2$. By homoclinicity, there exist $l\leq m\in \Z$ so that $d(f^n(y),f^{n}(p))<\alpha$ for $n\geq m+1$; $d(f^{n}(y),f^{n+s}(p))<\alpha$ for $n\leq l$ and $l+s\equiv m+1 \pmod {\per(p)}$. The sequence
                \[y_n=\begin{cases}
                    f^{l+n}(y), & 0\leq n\leq m-l,\\
                    f^{l+n}(p), & m-l+1\leq n\leq m-l+\per(p)
                    \end{cases}\]
            extended periodically to $\Z$ by $y_n=y_r$ for $r\equiv n \pmod {m-l+\per(p)+1}$ is an $\alpha$-pseudo-orbit. By the shadowing property, there exists $y'\in X$ that $\eps/2$-shadows $\{y_n\}_{n\in\Z}$. We verify that $y'$ is periodic. Define $M\coloneq m-l+\per(p)+1$, and let $n\in \Z$. Then 
                \begin{align*}
                    d\big(f^n(f^M(y')),f^n(y')\big)&\leq d\big(f^n(f^M(y')),y_{n+M}\big)+d\big(y_{n+M},f^n(y')\big) \\
                    &= d\big(f^n(f^M(y')),y_{n+M}\big)+d\big(y_{n},f^n(y')\big)<\eps<c_f,
                \end{align*}
            and by expansivity $f^M(y')=y'$. We denote this periodic point by $y'=p(y)$.

            Let $\mathcal{O}_{\text{hom}}$ be the finite set of homoclinic orbits in $Y$. We replace every homoclinic orbit $O(y)\in \mathcal{O}_{\text{hom}}$ with its corresponding periodic orbit $O(p(y))$ constructed as before to form $Y'$:
                \[Y'=\left(Y\setminus \bigcup_{O(y)\in\mathcal{O}_{\text{hom}}} O(y)\right)\cup \bigcup_{O(y)\in\mathcal{O}_{\text{hom}}} O(p(y)).\]

            Observe that every orbit in the first term of the union is either periodic or heteroclinic, and every $O(p(y))$ constructed as before is periodic. Therefore, $Y'$ has finitely many periodic orbits and finitely many heteroclinic orbits. 

            Since $Y$ is closed, and we have only replaced finitely many homoclinic orbits with periodic orbits that remain at distance strictly less than $\eps/2$ from  the closure of their original orbits, $Y'$ is a finite union of closed sets. Thus, $Y'$ is closed and $f$-invariant, meaning $Y'\in S(X,f)$. By our choice of sufficiently large $m$ and $l$, it follows that $d_{\mathcal{H}}(Y,Y')<\eps/2$. 

            In the second step, we construct an NMC subsystem $Y''\in S(X,f)$ satisfying $d_{\mathcal{H}}(Y',Y'')<\eps/2$. Consider the directed multi graph $G$ associated to $Y'$ defined as follows: the vertex set $V(G)$ is the set of periodic orbits of $Y'$, and the edge set $E(G)$ is the set of heteroclinic orbits in $Y'$. For every heteroclinic orbit $O$ with $\alpha(O)=P$ and $\omega(O)=Q$, we draw an edge from $P$ to $Q$. By the structure of $Y'$, it must be $|V(G)|, |E(G)|<\infty$. Note that $G$ has no self-loops since $Y'$ has no homoclinic orbits.

            We obtain $Y''$ by shadowing concatenations of finite orbit segments according to the graph $G$. For every edge $e\in E(G)$ we select a specific path in $G$ that is either a simple cycle or a simple barbell, i.e., two simple cycles connected by a directed simple transition path, and we shadow that specific path in $X$. We have the following cases:

            \begin{enumerate}[font=\upshape]
                \item If $e$ is contained in a strongly connected component (SCC) of $G$ that is a sink or source, we consider a simple cycle $K_e$ contained in the SCC and containing $e$. 
                \item If $e$ is contained in a strongly connected component of $G$ that is neither a sink nor a source, we extend the path forward until we traverse the closest exit of the SCC, and backwards until we traverse the closest entry to the SCC. We then follow a simple path to the next SCC and repeat the procedure until we enter a sink SCC. We extend the path with a simple cycle contained in the sink. Similarly we extend backward, ending in a simple cycle contained in a source SCC. This forms a simple barbell $B_e$ containing $e$. 
                \item If $e$ is in a transient component of $G$, we extend forward and backward as before to obtain a simple barbell $B_e$.
            \end{enumerate}
            Since $|E(G)|<\infty$, we obtain finitely many simple cycles $K_e$ and barbells $B_e$.
            For each simple cycle $K_e$, we construct a periodic orbit in $X$ as follows: let $P_1e_1P_2\ldots P_s e_s P_1$ be the simple path $K_e$ with $e=e_1$. Each $e_i$ corresponds to a heteroclinic orbit $O_i$ with $\alpha(O_i)=P_i$ and $\omega(O_i)=P_{i+1}$ if $i\leq s-1$, and $\omega(O_i)=P_1$ if $i=s$.  Given $\alpha\in (0,\eps/2)$ as before in the shadowing property, we can find points $y_i\in O_i$, integers $l_i\leq m_i$ and periodic points $p_i\in P_i$ satisfying the following for $1\leq i \leq s$:
                \[d\big(f^{n}(p_i),f^{n}(y_i)\big)<\alpha \,\forall\,n\leq l_i, \quad l_i \equiv 0 \pmod {\per(p_i)},\]
                \[d\big(f^n(p_{i+1}),f^{n}(y_i)\big)<\alpha \,\forall\,n\geq m_i+1, \; m_i+1\equiv 0 \pmod {\per(p_{i+1})}\]
                if $1\leq i \leq s-1$, and
                \[d(f^n(p_1),f^n(y_s))<\alpha \,\forall\,n\geq m_s+1, \quad m_s+1\equiv 0 \pmod {\per(p_1)}.\] 
            Define $M_i\coloneq \sum_{j=1}^i (\per(p_j)+m_j-l_j+1)$. The sequence 
                \[z_n=\begin{cases}
                        f^n(p_1), & 0\leq n \leq \per(p_1)-1\\
                        f^{n+l_1-\per(p_1)}(y_1), & \per(p_1)\leq n \leq M_1-1 \\
                        f^{n-M_{i-1}}(p_i), & M_{i-1}\leq n \leq M_{i-1}+\per(p_i)-1 \\
                        f^{n+l_i-\per(p_i)-M_{i-1}}(y_i), & M_{i-1}+\per(p_i)\leq n\leq M_i-1
                       \end{cases}\]
            for $2\leq i \leq s$, extended periodically by $z_n=z_r$ if $n\equiv r \pmod {M_s}$ is an $\alpha$-pseudo-orbit. Let $z_{K_e}\in X$ be the point that $\eps/2$-shadows $\{z_n\}_{n\in\Z}$. The same argument as in the homoclinic case shows that $z_{K_e}$ is periodic. 

            For each barbell $B_e$, we construct a biasymptotic orbit in $X$ as follows: let $P_1e_1\ldots P_se_sP_1$ be the initial simple cycle of $B_e$, let $P_{s+1}e_{s+1}P_{s+2}e_{s+2}\ldots e_{r-1} P_r$ be the transition path with $P_1=P_{s+1}$, and $P_re_{r}P_{r+1}\ldots P_{t}e_t P_{t+1}$ with $P_{t+1}=P_r$ be the terminal simple cycle. Denote by $O_i$ the heteroclinic orbit corresponding to the edge $e_i$ for $1\leq i \leq t$, with $\alpha(O_i)=P_i$ and $\omega(O_i)=P_{i+1}$. By biasymptoticity, we can find points $y_i\in O_i$, periodic points $p_i\in P_i$ (where $p_s=p_1$ and $p_r=p_{t+1}$), and integers $l_i\leq m_i$ so that for $1\leq i \leq t$:
                \[d\big(f^{n}(p_i),f^{n}(y_i)\big)<\alpha \,\forall\,n\leq l_i, \quad l_i \equiv 0 \pmod {\per(p_i)},\]
                \[d\big(f^n(p_{i+1}),f^{n}(y_i)\big)<\alpha \,\forall\,n\geq m_i+1,\; m_i+1\equiv 0 \pmod {\per(p_{i+1})}\]
                for $1\leq i\leq t-1$, and
                \[d(f^n(p_r),f^n(y_t))<\alpha \,\forall\,n\geq m_t+1, \quad m_t+1\equiv 0 \pmod {\per(p_r)}.\] 
            
            Define $M_i\coloneq \sum_{j=1}^i (\per(p_j)+m_j-l_j+1)$ for $i\in \{1,\ldots,t\}$. For $2\leq i\leq t$, define the sequence
                \[z_n=\begin{cases}
                        f^n(p_1), & 0\leq n \leq \per(p_1)-1\\
                        f^{n+l_1-\per(p_1)}(y_1), & \per(p_1)\leq n \leq M_1-1 \\
                        f^{n-M_{i-1}}(p_i), & M_{i-1}\leq n \leq M_{i-1}+\per(p_i)-1 \\
                        f^{n+l_i-\per(p_i)-M_{i-1}}(y_i), & M_{i-1}+\per(p_i)\leq n\leq M_i-1
                       \end{cases}.\]      
            Extend this sequence recursively by $z_n=z_p$ for the unique $p\in \{0,\ldots,M_s\}$ so that $n\equiv p \pmod {M_s}$, if $n<0$; and $z_n=z_q$ for the unique $q\in \{0,\ldots, M_t-M_r\}$ so that $n-M_t \equiv q \pmod {M_t-M_r}$ if $n\geq M_t$. The resulting sequence is an $\alpha$-pseudo-orbit. By the shadowing property, there exists $z_{B_e}\in X$ that $\eps/2$-shadows $\{z_n\}_{n\in\Z}$. We claim that $z_{B_e}$ is biasymptotic to two periodic orbits. 
            
            Recall that the periodic $\alpha$-pseudo orbits corresponding to the simple cycles $P_1e_1\ldots P_se_sP_1$ and $P_re_r\ldots P_te_tP_r$, say $z_n^i$ and $z_n^t$, constructed as in the previous case, are $\eps/2$-shadowed by two periodic points, say $z_i$ and $z_t$ respectively. Note that $z_n=z_{n}^i$ for $n<0$ and $z_n^t=z_{n+M_r}$ for $n\geq 0$. The point $z_{B_e}$ is biasymptotic to them; specifically, $\alpha(z_{B_e})=O(z_i)$ and $\omega(z_{B_e})=O(z_t)$. To prove this, we require a further restriction on $\eps$. In a Smale space $X$, there exists $\delta\in (0,\eps_X)$ so that $V^s(x)=\{y\in X \mid d(f^n(x),f^n(y))<\delta \, \forall n\geq 0\}$ and $V^u(x)=\{y\in X \mid d(f^n(x),f^n(y))<\delta \, \forall n\leq 0\}$. Assume that $\eps<\delta$, so  
                \begin{align*} d(f^n(f^{M_r}(z_{B_e})), f^n(z_t))&\leq d(f^{n+M_r}(z_{B_e}),z_{n+M_r})+d(z_{n+M_r},f^n(z_t))\\ &<\dfrac{\eps}{2}+d(z_{n}^t,f^n(z_t))<\eps<\delta \, \forall n\geq 0.
                \end{align*}
            which implies that $f^{M_r}(z_{B_e})\in V^s(z_t)$ so $\lim_{n\to\infty} d(f^{n+M_r}(z_{B_e}),f^n(z_t))=0$. This demonstrates that $O(z_{B_e})$ is forward asymptotic to $O(z_t)\subseteq \Per(X)$. Symmetrically, $O(z_{B_e})$ is backward asymptotic to $O(z_i)\subseteq \Per(X)$. 

            We now verify that 
                \[Y''\coloneq \bigcup_{\stackrel{e\in \text{SCC(G)}}{\text{sink or source}}}O(z_{K_e})\cup \bigcup_{\stackrel{e\not\in \text{SCC(G)}}{\text{sink or source}}}O(z_{B_e})\] 
            is an NMC subsystem of $(X,f)$. This $Y''$ is clearly $f$-invariant since it is a union of orbits; it is closed, being a finite union of closed sets; and it contains finitely many orbits, which are either periodic or biasymptotic to two periodic orbits. It remains to show that the forward limiting orbit of any heteroclinic orbit does not coincide with the backward limiting orbit of any other heteroclinic orbit.

            If we require $\eps<\min\{d(O,p) \mid O \text{ heteroclinic orbit in } Y', p\in\Per(Y')\cap\overline{O}^c \}$ and $\eps<\min\left\{d(p,q) \mid p\neq q\in\Per(Y')\right\}$, then $\eps>0$ since we have finitely many elements in each set over which we take the minimum. Assume $\eps$ is small enough to satisfy these bounds as well, and consider any two heteroclinic orbits $O$ and $O'$ in $Y''$ with limits $\alpha(O)=P$ and $\omega(O')=Q$, where $P,Q\subseteq \Per(Y'')$. Let $p=z^i(O)$ and $q=z^t(O')$ be the periodic points shadowing the initial and terminal simple cycles $P_1e_1\ldots e_s P_1$ and $Q_{r'}e'_{r'}\ldots e'_{t'}Q_{r'}$ associated to $O$ and $O'$ respectively. We will show that $O(p)\neq O(q)$.

            If $q\in O(p)$, then there exists $m\in\{0,\ldots, \per(q)-1\}$ so that $f^m(q)=p$. Note that $P_i\neq Q_j$ for any $i,j$ since $P_i$ belongs to a source SCC of $G$ and $Q_j$ belongs to a sink SCC of $G$ (they coincide only if $G$ has a unique SCC, in which case every edge has an associated simple cycle, no heteroclinic orbits are constructed, and NMC holds vacuously). Denote by $p_1\in P_1$ the periodic point so that $d(p_1,p)<\eps/2$. Then,
                \begin{align*} 
                    d(f^n(p_1),z_{n+m}^t(O'))&=d(z_{n}^i(O),z_{n+m}^t(O')) \\&\leq d(z_n^i(O),f^n(p))+d(f^n(p),z_{n+m}^t(O'))\\
                    &\leq \frac{\eps}{2}+d(f^{n+m}(q),z_{n+m}^t(O'))<\eps
                \end{align*}
            for $0\leq n\leq \per(p_1)-1$, which contradicts our restriction on $\eps$: $z_{n+m}^t(O')$ is either a periodic point in $Q_j$ for some $j$ or a point in a heteroclinic orbit corresponding to an edge $e'_j$, and $p_1$ is either a different periodic point in $Y'$ or at distance at least $\eps$ from the closure of such an orbit given that $e'_j$ and $P_1$ are in different SCCs of $G$. Therefore, it must be that $O(q)\neq O(p)$, ensuring that $Y''$ is NMC.

            It is clear from the construction that $d_{\mathcal{H}}(Y',Y'')<\eps/2$, so $d_{\mathcal{H}}(Y,Y'')<\eps$. This concludes the proof of the lemma.
        \end{proof}    

        The argument above shows that the family of NMC systems is dense in $S(X,f)$. 
        It remains to prove that NMC systems are isolated in $S(X,f)$. The strategy consists of using discreteness of subshifts with the NMC property in $S(\Sigma,\sigma)$ and transferring it to $S(X,f)$ via the induced map $\pi_*$. First, we need to prove that the preimage of an NMC system has the NMC property. The characterization of NMC, independent of Rauzy graphs, that we use is given below. It can be found in Lemma 3.5 of \cite{PavlovSchmieding}.

        \begin{definition}[NMC property for subshifts]\label{def:NMCsubshift}
            A subshift $Z$ has NMC if and only if there exist finite sets $P$, $M$, $S$ of words so that every sequence in $Z$ is either of the form $p^\infty$, $s^\infty$, or $p^\infty ms^\infty$ for some $p\in P$, $m\in M$, and $s\in S$, and so that no sequence can be written both as $p^\infty$ and $s^\infty$ for some $p\in P$ and $s\in S$.
        \end{definition}

        \begin{lemma}
            Let $Y\in S(X,f)$ be an NMC subsystem. Then every $Z\in \pi_*^{-1}(Y)$ has the NMC property.
        \end{lemma}

        \begin{proof}
            Let $Z$ and $Y$ be as in the statement of the lemma. The subshift $Z$ has the NMC property if and only if $Z$ contains finitely many periodic points; every aperiodic orbit is biasymptotic to two different periodic orbits; and the transition words between such periodic orbits are finite. This is just a different way of writing Definition \ref{def:NMCsubshift}.
            
            Since $\pi$ is finite-to-one, $\pi(z)\in Y$ is periodic if and only if $z\in Z$ is periodic, so the finiteness of $\Per(Z)$ is guaranteed. Now let $z\in Z$ be an aperiodic point and let $y=\pi(z)$ be its aperiodic image in $Y$. We prove that $z$ is biasymptotic to two different periodic orbits.

            Consider the sequence $\left\{\sigma^n(z)\right\}_{n\in \N}$ of forward iterates of $z$. Every accumulation point $w$ of this sequence must be a periodic point: 
            \[\sigma^{n_k}(z)\to w \Longrightarrow f^{n_k}(y)=f^{n_k}(\pi(z))=\pi(\sigma^{n_k}(z))\to \pi(w)\]
            thus $\pi(w)\in \omega(y)\subseteq \Per(Y)$, and consequently, $w\in \Per(Z)$. Furthermore, $\omega(z)\subseteq \pi^{-1}(\omega(y))$ is a finite, compact and invariant set. Since $\omega(z)$ is chain transitive, it cannot decompose into more than one periodic orbit. Hence $\omega(z)=O(p)$ for some $p\in\Per(Z)$. 
            We claim that \(z\) is forward asymptotic to \(O(p)\), i.e., $d(\sigma^n(z),O(p))\to 0$. Suppose, for a contradiction, that $d(\sigma^n(z),O(p))$ does not then to zero. Then there exists \(\varepsilon>0\) and an increasing sequence $\{n_k\}_{k\in\N}\subset\mathbb N$ such that
                \[ d(\sigma^{n_k}(z),O(p))\ge \varepsilon \text{ for every }k\in\N.\]
            By compactness of \(Z\), after passing to a subsequence we may assume that $\sigma^{n_k}(z)\to w\in Z$. Since \(n_k\to\infty\), it must be $w\in\omega(z)$. 
            Since the map $x\mapsto d(x,O(p))$ is continuous and vanishes on \(O(p)\), we obtain that 
                \[ d(\sigma^{n_k}(z),O(p))\rightarrow d(w,O(p))=0,\]
            contradicting the fact that \(d(\sigma^{n_k}(z),O(p))\ge\varepsilon\) for every \(k\). Hence \(d(\sigma^n(z),O(p))\rightarrow 0\), that is, \(z\) is forward asymptotic to the periodic orbit \(O(p)\).
            
            Similarly, $z$ is backward asymptotic to a single periodic orbit $O(q)\subseteq\pi^{-1}(\alpha(y))$.

            If $\alpha(z)=\omega(z)=O(p)$, the continuity and semiconjugacy of $\pi$ would give
                \[\alpha(\pi(z))=\pi(\alpha(z))=\pi(O(p))=\pi(\omega(z))=\omega(\pi(z)),\]
            contradicting the fact that $Y$ is an NMC subsystem. Therefore, $z$ is biasymptotic to two different periodic orbits.
            
            It remains to show that the transition words are finite, or equivalently, bounded in length.

            Assume, for the sake of contradiction, that the lengths of the transition words in $Z$ are not uniformly bounded. Then there exist periodic points $p^\infty,q^\infty \in \Per(Z)$ and a sequence \(x_n = p^\infty . w_n q^\infty \in Z\) such that $|w_n| \to \infty$ and $w_n$ neither starts with $p$ nor ends with $q$.
            Passing to a subsequence, assume $x_n \to x \in Z$. We fix the zero-coordinate at the position to the right of the dot. We first verify that the limit $x$ is aperiodic. By hypothesis, each $w_n$ is a minimal transition word, meaning it does not start with the periodic block $p$. Thus, for every $n \in \mathbb{N}$, ${x_n}_{[0, \lvert p \rvert)} \neq p$. Since this condition defines a closed set in the product topology, the limit must satisfy the same constraint, $x_{[0, \lvert p \rvert)} \neq p$. This rules out the possibility that $x = p^\infty$. Because $x$ retains the periodic left ray $p^\infty$ but is not globally periodic, it is therefore aperiodic.
            By the structural decomposition of $Z$, the point $x$ is biasymptotic to periodic orbits. Hence there exist $q' \in \mathrm{Per}(Z)$ and a finite word $w$ such that the future of $x$ satisfies \(x_{[0,\infty)} = w(q')^\infty\). Choose $w$ minimal among those representations, so that any other possible $w'r^\infty$ tail must have $|w'|\geq|w|$. After passing to a further subsequence, assume that $|w_n|>|w|+n|q'|$. By convergence to $x$, for every $m \in \mathbb{N}$ there exists $N(m)$ such that for all $n \ge N(m)$, 
                \[ {x_n}_{[|w|,\, |w|+m|q'|)}= (q')^m. \]
            Define \( k_n \coloneq \max \left\{ m \in \mathbb{N} : {x_n}_{[|w|,\, |w|+m|q'|)} = (q')^m \right\}\). 
            If $k_n$ were infinite for some $n$, it would mean that ${x_{n}}_{[0,\infty)}=w(q')^\infty$ and ${x_n}_{[0,\infty)}=w_nq^\infty$ so $O(q)=O(q')$. Hence, there exists $s$ with $|s|<|q|=|q'|$ such that $sq^\infty=(q')^\infty$. This would allow us to rewrite $x_n=p^\infty.wsq^\infty$ with $ws$ a transition word of length $|w|+|s|<|w|+|q|\leq |w_n|$, which contradicts the minimality of $w_{n}$ in the initial choice of $x_n$. We conclude that the maximum $k_n$ is well defined, and the preceding argument shows that $k_n \to \infty$.

            Define \(z_n \coloneq \sigma^{|w| + k_n|q'|}(x_n)\), which inherits ${z_n}_{[0,|q'|)}\neq q'$. By compactness, we extract a subsequence $z_n \to z\in Z$, and the left tail of $z$ satisfies $z_{(-\infty,0)} = (q')^\infty$ and $z_{[0,|q'|)}\neq q'$.
            If $z$ were periodic, its left tail $z_{(-\infty,0)} = (q')^\infty$ would force $z = (q')^\infty$. This immediately implies $z_{[0,|q'|)} = q'$. However, because $z_n \to z$, the closed condition ${z_n}_{[0,|q'|)} \neq q'$ is preserved in the limit, meaning $z_{[0,|q'|)} \neq q'$. This contradicts the construction of $z$, hence $z$ must be aperiodic.

            Thus, there exist aperiodic points $x,z \in Z$ such that \( \omega(x) = O(q') = \alpha(z)\). Since \(\omega(\pi(x)) = \pi(\omega(x))\) and \(\alpha(\pi(z)) = \pi(\alpha(z))\), we obtain \(\omega(\pi(x)) =\pi(O(q'))= \alpha(\pi(z))\). That is, we have found aperiodic points $\pi(x), \pi(z)\in Y$ whose $\omega$-limit and $\alpha$-limit, respectively, coincide. 
            This contradicts the fact that $Y$ is an NMC subsytem. Therefore, the assumption $|w_n| \to \infty$ is false, and the transition words in $Z$ are uniformly bounded.
        \end{proof}

        We have established the essential fact required to prove discreteness of NMC subsystems.

        \begin{lemma}\label{lem:NMCdiscrete}
            Every NMC subsystem $Y\in S(X,f)$ is isolated and, consequently, the family of NMC systems in $S(X,f)$ is discrete.
        \end{lemma}

        \begin{proof}
            Let $Y$ be an NMC subsystem and denote by $\text{NMC}(\Sigma,\sigma)$ the family of subsystems of $\Sigma$ with the NMC property. By the previous lemma, $\pi_*^{-1}(Y)\subseteq \text{NMC}(\Sigma,\sigma)$, which is a discrete set according to Theorem 3.6 in \cite{PavlovSchmieding}.
        
            We proceed by contradiction. If $Y$ were not isolated, there would be a sequence of different subsystems $Y_n\neq Y$ such that $d_\mathcal{H}(Y_n,Y)\to 0$. Consider the sequence of subsystems $\pi^{-1}(Y_n)\in S(\Sigma,\sigma)$ and note that they do not belong to $\pi_*^{-1}(Y)$, because surjectivity yields $\pi(\pi^{-1}(Y_n))=Y_n\neq Y$. However, $S(\Sigma,\sigma)$ is compact, so there exists a convergent subsequence. After passing to a subsequence and relabeling, we may assume is the same sequence $(\pi^{-1}(Y_n))_n$ converging to $Z\in S(\Sigma,\sigma)$. Therefore, $\pi_*(\pi^{-1}(Y_n))\to\pi_*(Z)$, i.e., $Y_n\to \pi_*(Z)$ and thus $\pi_*(Z)=Y$. This proves that there exists $Z\in\pi_*^{-1}(Y)$ that is not isolated, a contradiction with the fact that $\text{NMC}(\Sigma)\ni Z$ is discrete.  
        \end{proof}

        We have just proved that NMC subsystems in $S(X,f)$ are isolated. We can strengthen this claim to every isolated point.

        \begin{corollary}\label{cor:isolated}
            Every isolated point of $S(X,f)$ is an NMC subsystem.
        \end{corollary}

        \begin{proof}
            Let $Y\in S(X,f)$ be isolated. Since NMC subsystems are dense in $S(X,f)$, they intersect every open ball. In particular, they intersect the ball that equals the singleton $\{Y\}$. 
        \end{proof}

        Lemmas \ref{lem:NMCdense} and \ref{lem:NMCdiscrete} prove that a typical system in $S(X,f)$ is an NMC subsystem. We record some properties that a typical system in $S(X,f)$ must satisfy as a consequence of Theorem \ref{thm:NMC}.

        \begin{corollary}
            A generic subsystem in $S(X,f)$ is topologically conjugate to a countable sofic subshift, has at most linear complexity, and has zero topological entropy. Furthermore, if it contains more than one periodic orbit, it is not topologically transitive.
        \end{corollary}

        \begin{proof}
            Let $Y\in S(X,f)$ be a generic subsystem. By Theorem \ref{thm:NMC}, generic subsystems are NMC. Thus, $Y$ is the image of a subshift $Z\in S(\Sigma,\sigma)$ with the NMC property under a continuous factor map $\pi\colon Z\to Y$.

            By \cite{PavlovSchmieding}, the subshift $Z$ is a subshift of finite type (SFT) and has linear complexity if the set of its transition words is nonempty, and bounded complexity otherwise. Let $Y'$ be a subshift conjugate to $Y$, which exists since $Y$ is countable, compact and expansive (\cite{Mane}). Since $Y$ is a factor of $Z$, it follows that $Y'$ is a symbolic factor of the SFT $Z$, which immediately implies that $Y'$ is a sofic subshift. 

            Because the word complexity of a factor is bounded by the word complexity of the extension, the sofic subshift $Y'$ inherits at most linear complexity from $Z$. Sofic subshifts with at most linear complexity have zero topological entropy, so $h(\sigma,Y') = 0$. Because topological entropy is invariant under topological conjugacy, $h(f,Y) = 0$.

            Now, assume for the sake of contradiction that $Y$ has more than one periodic orbit and is topologically transitive. Since $Y$ is a compact metric space, topological transitivity is equivalent to the existence of a point $y \in Y$ with a dense forward orbit, meaning $\overline{O^+(y)} = Y$.
            By the NMC property exhibited by $Y$, the forward orbit of every point is asymptotic to a single periodic orbit. Let $\omega(y) = O(q)$ for some $q \in \Per(Y)$. The closure of the forward orbit of $y$ is exactly its forward orbit together with its $\omega$-limit set:
                \[ Y = \overline{O^+(y)} = O^+(y) \cup O(q).\]
            By hypothesis, $Y$ contains at least two distinct periodic orbits. Let $O(p) \subseteq Y$ be a periodic orbit such that $O(p) \neq O(q)$. Since $O(p)$ and $O(q)$ are disjoint, it must be that $O(p) \subseteq O^+(y)$. However, if a periodic orbit is contained in the forward orbit of $y$, the point $y$ itself must be periodic. This implies $Y = \overline{O^+(y)} = O(p)$, which contradicts the assumption that $Y$ contains multiple distinct periodic orbits. 
            Therefore, $Y$ cannot be topologically transitive.
        \end{proof}

    \section{The space of infinite topologically transitive systems \texorpdfstring{$T'(X,f)$}{T'(X,f)}}\label{section:toptransitive}

    The results on $S(X,f)$ in the previous section illustrate the simplicity and degeneracy of a typical system in $S(X,f)$. In order to avoid such dynamical behavior, we consider the space $T(X,f)$ of topologically transitive subsystems of $(X,f)$. If we remove all finite systems, we obtain the space $T'(X,f)$ of all infinite topologically transitive subsystems:
    \[T'(X,f)\coloneq \{Y\in S(X,f) \mid |Y|=\infty,\, Y \text{ is transitive}\}.\]
    Note that the definitions of $T(X,f)$ and $T'(X,f)$ make sense for any topological dynamical system $(X,f)$.

    By Corollary \ref{cor:isolated} and the fact that an NMC subsystem is topologically transitive if and only if it is a single periodic orbit, $T'(X,f)$ can be described as the derived set of $T(X,f)$. In other words, a topologically transitive system is isolated in $S(X,f)$ if and only if it is finite. 

    Recall from Proposition \ref{prop:contsurjpi} that the map $\pi_*\colon S(\Sigma,\sigma)\to S(X,f)$ is continuous and surjective. Consequently, $\pi_*\big|_{T'(\Sigma,\sigma)}\colon T'(\Sigma,\sigma)\to S(X,f)$ is also continuous.

    \begin{lemma}\label{lem:transitivesurj}
        The image of $T'(\Sigma,\sigma)$ under the induced map $\pi_*$ is $T'(X,f)$. Furthermore, the restricted map $\pi_*\big|_{\overline{T'(\Sigma,\sigma)}}\colon\overline{T'(\Sigma,\sigma)}\to \overline{T'(X,f)}$ is well-defined, continuous, and surjective.
    \end{lemma}

    \begin{proof}
        First, we prove that $\pi_*(T'(\Sigma,\sigma))\subseteq T'(X,f)$. Let $Z\in T'(\Sigma,\sigma)$ and $z$ be a transitive point. Thus, $\{\sigma^n(z)\}_{n\in\N}\stackrel{\pi}{\longmapsto}\{f^n(\pi(z))\}_{n\in\N}\subseteq \pi(Z)$, which implies that $\overline{\{f^n(\pi(z)) \mid n\in\N\}}\subseteq \pi_*(Z)$. On the other hand, for every $y\in \pi_*(Z)$ there exist $z'\in Z$ and an increasing sequence $\{n_k\}_{k\in\N}$ such that $\pi(z')=y$ and $\sigma^{n_k}(z)\to z'$. Therefore, the semiconjugacy gives
            \[\pi\circ \sigma^{n_k}(z)=f^{n_k}(\pi(z))\to \pi(z')=y,\]
        which proves that $y\in \overline{\{f^n(\pi(z)) \mid n\in\N\}}$, and so $\pi(z)$ is a transitive point for $\pi_*(Z)$, i.e., $\pi_*(Z)\in T'(X,f)$, because $\pi$ is finite-to-one so $\pi_*(Z)$ is infinite as well.

        For the reverse inclusion $T'(X,f)\subseteq \pi_*(T'(\Sigma,\sigma))$, let $Y\in T'(X,f)$ and $y\in Y$ be a transitive point. There exists $z\in \pi^{-1}(Y)\in S(\Sigma,\sigma)$ such that $\pi(z)=y$. Define $Z=\overline{\{\sigma^n(z) \mid n\in\N\}}\subseteq \pi^{-1}(Y)$, which is clearly closed, $\sigma$-invariant, infinite and has a transitive point, so $Z\in T'(\Sigma,\sigma)$. 

        Let $y'\in Y$, which we need to prove is the image under $\pi$ of a point in $Z$. 
        If $y'$ is such that $y'=f^n(y)$ for some $n\in\N$, then $y'=\pi(\sigma^n(z))$ and so $y'\in \pi_*(Z)$. If that is not the case, it must be that $y'=\lim_{k\to\infty}f^{n_k}(y)$ with $\{n_k\}$ an increasing subsequence of $\N$. Consider the sequence $\{\sigma^{n_k}(z)\}_{k\in\N}$. By compactness of $Z$, there exists a subsequence $n_{k_j}$ such that $\sigma^{n_{k_j}}(z)\to z'\in Z$. The semiconjugacy gives that $\pi(z')=y'$, and therefore $\pi_*(Z)=Y$. This concludes the proof of the first statement.  

        For the remaining statement, $\pi_*(T'(\Sigma,\sigma))=T'(X,f)$ and the continuity of $\pi_*$ imply that $\pi_*(\overline{T'(\Sigma,\sigma)})\subseteq\overline{\pi_*(T'(\Sigma,\sigma))}=\overline{T'(X,f)}$, thus the restriction map is well-defined and continuous by Proposition \ref{prop:contsurjpi}. 
        In addition, $\pi_*(\overline{T'(\Sigma,\sigma)})$ is a closed set containing $T'(X,f)$, which implies that $\pi_*(\overline{T'(\Sigma,\sigma)})\supseteq \overline{T'(X,f)}$ by the definition of closure of a set. This concludes the lemma. 
        \end{proof}

        Genericity in this section will be stated for $\overline{T'(X,f)}$ instead of $T'(X,f)$, because the former is complete and the Baire Category Theorem applies there, whereas the latter is not known a priori to be completex. However, genericity is equivalent in both sets, which follows from the result below.

        \begin{lemma}\label{lem:TdenseGdelta}
            Let $(X,f)$ be an irreducible Smale space. Then $T'(X,f)$ is a $G_\delta$ dense set in its closure $\overline{T'(X,f)}$.
        \end{lemma}
        
        \begin{proof}
            Since $X$ is a compact metric space, it is second countable. Let $\{U_n\}_{n\in\N}$ be a countable basis of open sets for $X$. A subsystem $Y\in S(X,f)$ is topologically transitive if and only if for every $i,j\in\N$ with $Y\cap U_i\neq\emptyset$ and $Y\cap U_j\neq \emptyset$, there exists some $m\in \N$ such that 
                \[Y\cap U_i\cap f^{-m}(U_j)\neq\emptyset.\]
            For fixed indices $i,j\in\N$, consider the set $A(i,j)\coloneq \{Y\in S(X,f) \mid Y\cap U_i\neq\emptyset \text{ and } Y\cap U_j\neq\emptyset\}$, which is open in $(S(X,f),d_\mathcal{H})$. For a fixed $m\in\N$, define 
                \[ B(i,j,m)\coloneq \{Y\in S(X,f) \mid Y\in A(i,j), Y\cap U_{i}\cap f^{-m}(U_j)\neq\emptyset\}.\]
            We verify that $B(i,j,m)$ is open. Let $Y\in B(i,j,m)$ and choose $y\in Y\cap U_i\cap f^{-m}(U_j)$. Since $U_i\cap f^{-m}(U_j)$ is open in $X$, there exists $\eps>0$ such that $B(y,\eps)\subseteq U_i\cap f^{-m}(U_j)$. For any $Y'\in S(X,f)$ with $d_{\mathcal{H}}(Y,Y')<\eps$, the definition of the Hausdorff metric implies $\inf_{y'\in Y'}d(y,y')<\eps$. Because $Y'$ is closed, the infimum is attained, meaning there is $y'\in Y'$ such that $d(y,y')<\eps$. Thus, $y'\in B(y,\eps)\subseteq U_i\cap f^{-m}(U_j)$, which proves $Y'\cap U_i\cap f^{-m}(U_j)\neq\emptyset$. Since $Y'$ is $f$-invariant, this also implies $Y'\cap U_i\neq\emptyset$ and $Y'\cap U_j\neq\emptyset$, hence $Y'\in A(i,j)$. This proves that $B_{\mathcal{H}}(Y,\eps)\subseteq B(i,j,m)$, so the latter set is open.
            
            The set of all topologically transitive systems $T(X,f)$ is then characterized by:
                \[T(X,f)=\bigcap_{i,j\in\N} \left(A(i,j)^c \cup \bigcup_{m\geq 1}B(i,j,m)\right).\]
            For each pair $(i,j)$, $A(i,j)^c$ is closed and $\cup_{m\geq 1}B(i,j,m)$ is open. In a metric space, both open and closed sets are $G_\delta$. Since finite unions and countable intersections of $G_\delta$ sets are $G_\delta$, $T(X,f)$ is a $G_\delta$ set in $S(X,f)$.
            
            Now, let $\text{Iso}(S(X,f))$ denote the set of isolated points in $S(X,f)$. For each isolated point $Y$, the singleton $\{Y\}$ is open, meaning $\text{Iso}(S(X,f)) = \bigcup_{Y \in \text{Iso}S(X,f)} \{Y\}$ is an open set. Its complement,
                \[S'(X,f) = S(X,f) \setminus \text{Iso}(S(X,f)),\]
            is a closed set in $S(X,f)$, and is therefore a $G_\delta$ set. As discussed before, $T'(X,f)$ consists of the transitive systems that are not isolated points in $S(X,f)$, so we can write it as the intersection: 
                \[T'(X,f) = T(X,f) \cap S'(X,f).\]
            As the intersection of two $G_\delta$ sets, $T'(X,f)$ is a $G_\delta$ set in the ambient space $S(X,f)$. It follows immediately from the subspace topology that $T'(X,f)$ is a $G_\delta$ set relative to its closure $\overline{T'(X,f)}$. Since any set is dense in its own closure, $T'(X,f)$ is a dense $G_\delta$ set in $\overline{T'(X,f)}$.
        \end{proof}

        \begin{lemma}\label{lem:Tperfect}
            Let $(X,f)$ be an irreducible Smale space. The set $\overline{T'(X,f)}$ is a perfect subset of $S(X,f)$. 
        \end{lemma}

        \begin{proof}
            The set $\overline{T'(X,f)}$ is closed by definition, so it remains to prove that it has no isolated points.

            Let $Y\in\overline{T'(X,f)}$ and $\eps>0$ be fixed. By Lemma \ref{lem:transitivesurj}, $U\coloneq\pi_*^{-1}(B_\mathcal{H}(Y,\eps))\cap \overline{T'(\Sigma,\sigma)}$ is a nonempty open set in the relative topology of $\overline{T'(\Sigma,\sigma)}$. Let $Z\in U$ be an $n$-step SFT approximation of a subsystem $W\in U$, which is transitive if the subsystem $W$ is, with $n$ large enough so that $U_Z\coloneq B_{\mathcal{H}}(Z,2^{-n})\cap\overline{T'(\Sigma,\sigma)}\subseteq U$. Since this set is relatively open in $\overline{T'(\Sigma,\sigma)}$, the construction in Lemma 5.3 of \cite{PavlovSchmieding} implies that there exists $Z'\subseteq Z$ in $U_Z$ such that $Z'\neq Z$. The strict monotonicity of the topological entropy under proper inclusions for irreducible SFTs yields
                \[h(\sigma,Z)>h(\sigma,Z').\]
            Hence 
                \[\pi(Z), \pi(Z')\in\pi(U_Z)\subseteq \pi(U)=B_\mathcal{H}(Y,\eps)\cap\overline{T'(X,f)},\]
            but the topological entropy is preserved under finite-to-one maps so 
                \[h(\sigma,Z)=h(f,\pi(Z))>h(f,\pi(Z'))=h(\sigma,Z').\]
            Therefore $\pi(Z)$ and $\pi(Z')$ are two distinct elements in $B_\mathcal{H}(Y,\eps)$. Since $\eps>0$ and $Y$ are arbitrary, we conclude that $\overline{T'(X,f)}$ is perfect.
        \end{proof}
        
        \subsection{Genericity of topological properties}
        Now that we can use the Baire Category Theorem to produce residual sets in $\overline{T'(X,f)}$, we prove that minimality, unique ergodicity, and zero topological are generic. The technique to obtain a $G_\delta$ set is similar to the one presented in Lemma \ref{lem:TdenseGdelta} above, but tailored to the respective properties we focus on. 

        Minimality of a dynamical system $(Y,g)$ is often phrased in terms of transitive points: $(Y,g)$ is minimal if and only if every point $y\in Y$ has a dense positive orbit. Another classical characterization of minimality states that if $Y'\subseteq Y$ is compact and $g$-invariant, then either $Y'=\emptyset$ or $Y'=Y$. In this case, we use a variant of the definition that is valid on compact spaces, which can be stated in terms of open sets: $Y\in S(X,f)$ is minimal if and only if for every open set $U\neq \emptyset$ there exists $N\in\N$ such that $\bigcup_{n=0}^N f^{-n}(U)\supseteq Y$.
        
        A similar proof of the following result in the context of homeomorphisms of the Cantor set can be found in \cite{Hochman}.

        \begin{theorem}\label{thm:MinimalgenericinT}
            Let $(X,f)$ be an irreducible Smale space. The set of minimal subsystems $M(X,f)$ is generic in $\overline{T'(X,f)}$.
        \end{theorem}

        \begin{proof}
            Since $X$ is a compact metric space, it is second countable. Let $\{U_j\}_{j\in\N}$ be a countable basis of open sets for $X$. Using the previous characterization, a subsystem $Y\in S(X,f)$ is minimal if and only if for every $j\in \N$, either $Y\cap U_j=\emptyset$, or 
                \[ Y\cap U_j\neq\emptyset \Longrightarrow \exists N\in\N \text{ so that } \bigcup_{n=0}^Nf^{-n}(U_j)\supseteq Y.\]

            Fix $j\in\N$ and define the set $A(j)\coloneq \{Y\in S(X,f) \mid Y\cap U_j\neq\emptyset\}$, which is open in $(S(X,f),d_{\mathcal{H}})$. For every $N\in \N$, define 
            \[ B(j,N)\coloneq \left\{ Y\in S(X,f) \mid Y\in A(j) \text{ and } \bigcup_{n=0}^N f^{-n}(U_j)\supseteq Y\right\}.\]

            We verify that $B(j,N)$ is open in $S(X,f)$. Let $Y\in B(j,N)$, meaning $Y$ is contained in the open set $B_N\coloneq \bigcup_{n=0}^N f^{-n}(U_j)$. Since $Y$ is compact and the complement $B_N^c$ is closed, the distance $2\eps\coloneq d\left(Y,\partial B_N^c\right)$ is strictly positive (if $B_N=X$, any $\eps>0$ suffices).
            
            Consider any $Y'\in S(X,f)$ such that $d_{\mathcal{H}}(Y,Y')<\eps$. By the definition of the Hausdorff metric, $Y'$ is contained in the $\eps$-neighborhood of $Y$, denoted $Y_\eps\coloneq\{x\in X \mid d(x,Y)<\eps\}$. By our choice of $\eps$, we have $B_N \supseteq Y_\eps \supseteq Y'$. Furthermore, since $Y'\cap f^{-n}(U_j)\neq\emptyset$ for some $n\leq N$, it follows by $f$-invariance that $f^n(Y')\cap U_j=Y'\cap U_j\neq\emptyset$. Therefore, $B_{\mathcal{H}}(Y,\eps)\subseteq B(j,N)$, proving that $B(j,N)$ is open.

            The set of all minimal systems $M(X,f)$ is then characterized by
                \[M(X,f)=\bigcap_{j\in\N} \left(A(j)^c \cup \bigcup_{N\in\N}B(j,N)\right).\]

            For each $j\in \N$, $A(j)^c$ is closed and the union $\bigcup_{N\in\N}B(j,N)$ is open. Because open and closed sets in a metric space are $G_\delta$, their finite unions and countable intersections are also $G_\delta$. Therefore, $M(X,f)$ is a $G_\delta$ set in $S(X,f)$. 
            Consequently, the intersection $M(X,f)\cap \overline{T'(X,f)}$ is a $G_\delta$ set in the relative topology of $\overline{T'(X,f)}$.
            
            It remains to show the density of $M(X,f)\cap \overline{T'(X,f)}$. Note that the factor map $\pi\colon \Sigma\to X$ maps minimal systems in $\overline{T'(\Sigma,\sigma)}$ to minimal systems in $\overline{T'(X,f)}$. Since $M(\Sigma,\sigma)\cap\overline{T'(\Sigma,\sigma)}$ is dense in $\overline{T'(\Sigma,\sigma)}$, it follows by Lemma \ref{lem:transitivesurj} that $\pi_*(M(\Sigma,\sigma)\cap \overline{T'(\Sigma,\sigma)})\subseteq M(X,f)\cap \overline{T'(X,f)}$ is dense in $\overline{T'(X,f)}$ as well.
        \end{proof}

        Every subsystem $Y\in S(X,f)$ inherits the expansivity of $(X,f)$. In particular, every $Y\in M(X,f)$ is minimal and expansive. By the theorem in \cite{Mane}, a system with such properties is zero-dimensional and thus conjugate to a subshift.

        \begin{corollary}\label{cor:symbolic}
            Let $(X,f)$ be an irreducible Smale space. A typical system in $\overline{T'(X,f)}$ is conjugate to a subshift over a finite alphabet.
        \end{corollary}

        In the previous section, we considered a family of subshifts that was generic in $S(\Sigma,\sigma)$, which motivated the definition of a family of subsystems in $S(X,f)$ with roughly the same dynamical properties as its symbolic counterpart. 
        Similarly, the family of regular Toeplitz subshifts of $\Sigma$ is generic in the space $\overline{T'(\Sigma,\sigma)}$ (Theorem 5.11 in \cite{PavlovSchmieding}). Given that a typical system in $\overline{T'(X,f)}$ is symbolic, a natural question is whether the generic subsystem is still regular Toeplitz.

        Regular Toeplitz subshifts are, in a certain way, very simple and rigid systems. They arise as the symbolic, minimal, almost one-to-one extensions of odometers that satisfy a specific regularity. We define this condition below.

        \begin{definition}
            A sequence $x\in \mathcal{A}^\Z$, with $\mathcal{A}$ a finite set, is regular Toeplitz if there exists a periodic structure $p=\{p_m\}_{m\in\N}$ with $p_m\mid p_{m+1}$ such that
                \[\bigcup_{m\in\N} \mathrm{Per}_{p_m}(x)=\Z, \quad \mathrm{Per}_{p_m}(x)\coloneq\{n\in\Z \mid x_n=x_{n+kp_m}\,\forall k\in\Z\},\]  
            and such that $\lim_{m\to\infty} D({\mathrm{Per}}_{p_m}(x))=1$, where $D$ denotes the density of a subset of $\Z$. The orbit closure $\overline{O_\sigma(x)}$ is called a regular Toeplitz subshift.
        \end{definition}

        Within the space of homeomorphisms of the Cantor set, Hochman proved that the conjugacy class of the universal odometer is generic \cite{Hochman}. In the expansive setting, such as a subshift or, more generally, a Smale space $(X,f)$, odometers cannot exist as subsystems because they are always non-expansive. 

        In the symbolic setting, it was proved in \cite{PavlovSchmieding} that regular Toeplitz subshifts that extend the universal odometer are generic in the space of infinite transitive subsystems. These systems are, in a sense, well-behaved: they are always uniquely ergodic, have zero topological entropy and are metrically isomorphic to their maximal equicontinuous factor (MEF), which in that case is the universal odometer.

        In the general case of an irreducible Smale space $(X,f)$, the regular Toeplitz subshifts that factor onto the universal odometer are dense in $\overline{T'(X,f)}$.  This is the content of the following lemma.

        \begin{lemma}\label{lem:toeplitz}
            Let $(X,f)$ be an irreducible Smale space, and let $\pi_*\colon\overline{T'(\Sigma,\sigma)}\to \overline{T'(X,f)}$ be the induced map in Lemma \ref{lem:transitivesurj}. The image under $\pi_*$ of a regular Toeplitz subshift that factors onto the universal odometer is conjugate to a regular Toeplitz subshift that factors onto the universal odometer.
        \end{lemma}

        \begin{proof}
            We first prove that the factor of a Toeplitz subshift is a Toeplitz subshift, obtained as an extension of the universal odometer. We then prove that the image under $\pi$ of a regular Toeplitz sequence is also regular Toeplitz.
        
            Every factor of a regular Toeplitz subshift $Z\in S(\Sigma,\sigma)$ inherits the minimality of $Z$ and the expansivity of $(X,f)$. The theorem in \cite{Mane} implies  that $\pi_*(Z)$ is conjugate to a subshift $Y$. Let $\phi\colon\pi_*(Z)\to Y$ be such a conjugacy. 
            
            Every Toeplitz subshift is an almost one-to-one extension of an odometer, which in this case we denote by $(G_1,\tau_1)$. By Theorem 3.1 in \cite{DownDurand}, there exists a topological factor $(G_2,\tau_2)$ of $(G_1,\tau_1)$ such that $(Y,\sigma)$ is an almost one-to-one extension of $(G_2,\tau_2)$, i.e., the following diagram commutes: 
            
            \begin{center}
                \begin{tikzcd}
                {(Z,\sigma)} \arrow[rr, "\pi"] \arrow[dd, "p_Z"']  &  & {(\pi_*(Z),f)} \arrow[rr, "\phi"] &  & {(Y,\sigma)} \arrow[dd, "p_Y", dashed] \\
                                                                &  &                                   &  &                                        \\
                {(G_1,\tau_1)} \arrow[rrrr, "\tilde{\pi}", dashed] &  &                                   &  & {(G_2,\tau_2)}                        
                \end{tikzcd}
            \end{center}

            Since $(G_2,\tau_2)$ is a topological factor of $(G_1,\tau_1)$, it is also an odometer. Since $Y$ is minimal, an almost one-to-one extension of an odometer, and symbolic, it is a Toeplitz subshift.

            Observe that, in this case, we can assume that $(G_1,\tau_1)$ is the universal odometer. The induced surjective group homomorphism $\tilde{\pi}$ must be finite-to-one: the fiber of any $g_2\in G_2$ is a coset of $\ker(\tilde{\pi})$, so all fibers of $\tilde{\pi}$ have the same cardinality. Let $g_2\in G_2$ be a point with a unique preimage under $p_Y$, say $p_Y^{-1}(g_2)=\{y\}$. Since $\phi\circ\pi$ is finite-to-one, $(p_Y\circ\phi\circ\pi)^{-1}(g_2)$ is a finite set. If $\tilde{\pi}^{-1}(g_2)$ were infinite, then $(\tilde{\pi}\circ p_Z)^{-1}(g_2)$ would be infinite too, but the commutativity of the diagram implies that 
            $(p_Y\circ\phi\circ\pi)^{-1}(g_2)=(\tilde{\pi}\circ p_Z)^{-1}(g_2)$, a contradiction.
            This also proves that $\tilde{\pi}^{-1}(0)=\ker(\tilde{\pi})$ is a finite subgroup of the torsion-free universal odometer $(G_1,\tau_1)$, so it must be trivial. This proves that $\tilde{\pi}$ is a group isomorphism and, consequently, $(G_2,\tau_2)$ is the universal odometer.

            Now, we need to prove that $Y$ is a regular Toeplitz subshift. Recall the conjugacy $\phi\circ\pi\colon Z\to Y$ between subshifts. By the Curtis-Lyndon-Hedlund theorem, there exist $r\geq 0$ and a sliding block code $\varphi:\mathcal{L}_{2r+1}(Z)\to \mathcal{L}_1(Y)$ defining $\phi\circ\pi$ by $\phi\circ\pi(z)_i=\varphi(z_{i-r},\ldots,z_{i+r})$. Also, since $z\in Z$ is a regular Toeplitz sequence, there exists a periodic structure $p=\{p_i\}_{i\in\N}$ with $p_i\mid p_{i+1}$ such that 
            \[\bigcup_{m\geq 1} \Per_{p_m}(z)=\Z, \quad \text{ where } \Per_{p_m}(z)\coloneq \{n\in\Z \mid z_{n}=z_{n+kp_m} \,\forall k\in\Z\},\]
            and such that $\lim_{m\to\infty}D(\Per_{p_m}(z))=1$.

            Note that $z$ is aperiodic and the factor map is finite-to-one, so $\phi\circ\pi(z)\eqcolon y$ is also aperiodic. First we verify that $y$ is $p$-periodic:
            
            Let $n\in \Z$. By definition of regularity of $z$, for every $j\in [n-r,n+r]$ there exists $N_j$ such that for all $m\geq N_j$, $z_j=z_{j+kp_m} \forall k\in \Z$. Set $N_n=\max \{N_j \mid j\in [n-r,n+r]\}$. Then for $m>N_n$ 
                \[ z_{[n-r,n+r]}=z_{[n-r+kp_m,n+r+kp_m]} \, \forall k\in\Z.\]
            Applying the sliding block code of window size $2r+1$, we obtain, for every $m>N_n$
                \[ \phi\circ\pi(z)_n=\varphi\left(z_{[n-r,n+r]}\right)=\varphi\left(z_{[n-r+kp_m,n+r+kp_m]}\right)=\phi\circ\pi(z)_{n+kp_m} \,\forall k\in \Z.\]
            Therefore, $n\in \bigcup_{i\geq 1}\Per_{p_i}(y)$ and hence $\Z=\bigcup_{i\geq 1}\Per_{p_i}(y)$.

            It remains to show that $\lim_{m\to\infty}D(\Per_{p_m}(y))=1$. The previous argument proves that for $p_m\in p$ fixed
                \[ Q_m\coloneq \{n\in\Z \mid [n-r,n+r]\subseteq \Per_{p_m}(z)\}\subseteq \Per_{p_m}(y),\]

            which translates to the inequality of upper densities $\overline{D}(Q_m)\leq D(\Per_{p_m}(y))$, where the bar is omitted from the term on the right-hand side  because such a limit always exists by Theorem 2.8 in \cite{DownarowiczKasjan}.

            We estimate $\overline{D}(Q_{m})$, the upper density of $Q_m$, by studying its complement. Note that $n\in Q_{m}^c$ if and only if $\exists j\in [-r,r]$ such that $j+n\not\in \Per_{p_m}(z)$, i.e.,
                \[ Q_m^c\subseteq \bigcup_{j=-r}^r \big(\left(\Per_{p_m}(z)\right)^c-j\big) \Rightarrow |Q_m^c \cap [-N,N]|\leq \sum_{j=-r}^r |(\Per_{p_m}(z))^c\cap [-N,N]|\] 
            for every $N\in \N$. Dividing by $2N+1$ and taking $\liminf_{N\to\infty}$ gives
                \[1-D(\Per_{p_m}(y))\leq 1-\overline{D}(Q_m)=\underline{D}(Q_m^c)\leq (2r+1)(1-D(\Per_{p_m}(z))),\]
            which implies that
                \[1-(2r+1)+(2r+1)D(\Per_{p_m}(z))\leq D(\Per_{p_m}(y)).\]
            Letting $m\to\infty$ and using the fact that $\lim_{m\to\infty} D\left(\Per_{p_m}(z)\right)=1$, we conclude that $\lim_{m\to\infty} D(\Per_{p_m}(y))=1$.             
        \end{proof}

        The theorem below immediately follows from Lemma \ref{lem:transitivesurj}, Lemma \ref{lem:toeplitz} above and Theorem 5.15 in \cite{PavlovSchmieding}.
        
        \begin{theorem}\label{thm:Toeplitzdense}
            Let $(X,f)$ be an irreducible Smale space. The family of subsystems in $S(X,f)$ that are conjugate to a regular Toeplitz subshift that factors onto the universal odometer is dense in $\overline{T'(X,f)}$.
        \end{theorem}

        We call a dynamical system $(Y,g)$ uniquely ergodic if there exists a unique $g$-invariant Borel probability measure on $Y$. Denote by 
            \[\text{UE}(X,f)\coloneq \{Y\in S(X,f) \mid Y \text{ is uniquely ergodic}\}\]
        the set of subsystems of $(X,f)$ that are uniquely ergodic. We prove that $\text{UE}(X,f)$ is generic in $\overline{T'(X,f)}$. 

        The following proof relies on the standard fact that if $(Y,g)$ is uniquely ergodic with $g$-invariant measure $\mu$, then the ergodic averages $S_N \phi\coloneq \tfrac{1}{N}\sum_{j=0}^{N-1}\phi\circ g^j$ (corresponding to $g$) converge uniformly to $\int_Y \phi\,d\mu$ for every $\phi$ in the space of continuous real-valued functions $C(Y,\mathbb{R})$. 
        
        \begin{theorem}\label{thm:UEgenericinT}
            Let $(X,f)$ be an irreducible Smale space. The family of uniquely ergodic subsystems $\mathrm{UE}(X,f)$ is residual in $\overline{T'(X,f)}$.  
        \end{theorem}

        \begin{proof}
             Consider a countable collection $\{g_i\}_{i\in\N}$ of continuous functions that is dense in $C(X,\mathbb{R})$. Denote by $S_N g_i(x)$ the ergodic average at $x$ corresponding to $f$. By a standard density argument and Theorem 6.19 in \cite{Walters}, a subsystem $Y\in S(X,f)$ is uniquely ergodic if and only if
                \[\lim_{N\to\infty} \sup_{x\neq y \in Y}\left|S_Ng_i(x)-S_Ng_i(y)\right|=0 \quad \forall i\in \N.\]
            Equivalently, if we define
                \[O(i,n,N)\coloneq \left\{Y\in S(X,f) \mid \sup_{x\neq y \in Y}\left|S_Ng_i(x)-S_Ng_i(y)\right|<\dfrac{1}{n}\right\},\]
            then we can describe uniquely ergodic systems as
                \[\text{UE}(X,f)=\bigcap_{i\in\N}\bigcap_{n\in\N}\bigcup_{N\in\N} O(i,n,N).\]
            
            The set $O(i,n,N)$ is open in the Hausdorff metric: if $Y\in O(i,n,N)$, let 
                \[3\eta=\dfrac{1}{n}-\sup_{x\neq y \in Y}\left|S_Ng_i(x)-S_Ng_i(y)\right|.\] 
                
            Consider $\eps>0$ such that, for any $x,x'\in X$, $\left|g_i\circ f^k(x)-g_i\circ f^k(x')\right|<\eta$ for every $k\in \{0, 1,\ldots, N-1\}$ whenever $d(x,x')<\eps$. We verify that $B_\mathcal{H}(Y,\eps)\subseteq O(i,n,N)$.

            Let $Y'\in B_{\mathcal{H}}(Y,\eps)$. For every pair of points $x',y'\in Y'$ we can find $x,y\in Y$ so that $d(y,y'), d(x,x')<\eps$. Thus,
                \begin{align*}
                    \left|S_N g_i(x')-S_N g_i(y')\right|&\leq \left|S_N g_i(x')-S_N g_i(x)\right|+\left|S_Ng_i(x)-S_N g_i(y)\right|\\ &+\left|S_Ng_i(y)-S_Ng_i(y')\right| 
                    \leq \eta+\left|S_Ng_i(x)-S_N g_i(y)\right|+\eta \\
                    &<3\eta+\sup_{x\neq y\in Y}\left|S_Ng_i(x)-S_N g_i(y)\right|=\dfrac{1}{n},
                \end{align*}
            and, consequently, $Y'\in O(i,n,N)$.
            It follows that $\text{UE}(X,f)$ is a $G_\delta$ set of $S(X,f)$, which also implies that $\text{UE}(X,f)\cap \overline{T'(X,f)}$ is a $G_\delta$ set in the relative topology of $\overline{T'(X,f)}$.
            
            Regular Toeplitz subshifts are always uniquely ergodic (Theorem 2.5 in \cite{DownarowiczKasjan}). Since this property is invariant under topological conjugacy, the family of subsystems that was proved dense in $\overline{T'(X,f)}$ in Theorem \ref{thm:Toeplitzdense} satisfies unique ergodicity. Therefore, $\text{UE}(X,f)$ is also dense in $\overline{T'(X,f)}$.
        \end{proof}

        The family of subsystems of $(X,f)$ having zero topological entropy, which we denote by $\text{ZE}(X,f)$, turns out to be generic as well. A thorough treatment of topological entropy in the compact metric setting can be found in \cite{Walters}.

        In order to prove genericity of such a subsystem, we need to borrow tools from the measure-theoretical realm. The following result has an elementary proof that we provide for the sake of completeness.

        \begin{lemma}\label{lem:entropyuppersemi}
            Let $(X,f)$ be an irreducible Smale space. The entropy function $h\colon S(X,f)\to \mathbb{R}$, $Y\mapsto h(f,Y)$ is upper-semicontinuous.
        \end{lemma}

        \begin{proof}
            Upper-semicontinuity of $h$ is equivalent to the following statement: if $Y_n\to Y$ in the Hausdorff metric, then $\limsup_{n\to\infty}h(f,Y_n)\leq h(f,Y)$. 

            Let $\{Y_n\}_{n\in\N}\subseteq S(X,f)$ converge to some $Y\in S(X,f)$. For $n\in\N$ fixed, the variational principle for the entropy says that
                \[h(f,Y_n)=\sup\{h_{\mu_n} \mid \mu_n \,f\text{-invariant Borel probability, }\supp\mu_n\subseteq Y_n\},\]
            where $h_\mu$ denotes the metric entropy of the measure $\mu$. 

            Consider an increasing subsequence $n_j$ such that 
                \[\lim_{j\to\infty}h(f,Y_{n_j})=\limsup_{n\to\infty}h(f,Y_n).\] 
            For each $n_j$, the variational principle allow us to choose $\mu_{n_j}$ supported on $Y_{n_j}$ so that according to the variational principle $h_{\mu_{n_j}}>h(f,Y_{n_j})-\frac{1}{n_j}$. The weak-\textasteriskcentered\ compactness of the space of $f$-invariant measures on $X$ provides a subsequence convergent to some $f$-invariant $\mu$. Without loss of generality, denote this subsequence by $\mu_{n_j}$.

            The map $f$ is expansive, so the metric entropy (defined in the space of $f$-invariant Borel probability measures) is upper-semicontinuous in the weak-\textasteriskcentered\ topology. Consequently, $\limsup_{j\to\infty}h_{\mu_{n_j}}\leq h_\mu$.

            The limiting measure $\mu$ is supported on $Y$: let $y_0\in Y^c$ and $r>0$ such that $B(y_0,r)\subseteq Y^c$. Let $J\in\N$ be so that $d(Y_{n_j},Y)<\frac{r}{2}$ for every $j\geq J$. It must be that $B(y_0,r/2)\subseteq Y_{n_j}^c$ for every $j\geq J$. Let $g\in C(X,[0,1])$ such that $g\equiv 1$ on $B(y_0,r/4)$ and $g\equiv 0$ on $B(y_0,r/2)^c$. Then, $(\supp g)\cap Y_{n_j}=\emptyset$ and 
            
                \[0=\int g\,d\mu_{n_j}\xrightarrow{j\to\infty}\int g\,d\mu \,\Longrightarrow\, 0=\int g\,d\mu \geq \mu(B(y_0,r/4)),\]
            which proves that $y_0\in (\supp \mu)^c$.

            By the variational principle, $h_\mu\leq h(f,Y)$ and 
                \[\limsup_{n\to\infty}h(f,Y_n)=\lim_{j\to\infty}h(f,Y_{n_j})=\lim_{j\to\infty}h_{\mu_{n_j}}\leq h_\mu\leq h(f,Y),\]
            as desired.
        \end{proof}

        \begin{theorem}\label{thm:zeroentropy}
            Let $(X,f)$ be an irreducible Smale space. The family of subsystems with zero topological entropy $\mathrm{ZE}(X,f)$ is residual in $\overline{T'(X,f)}$.
        \end{theorem}

        \begin{proof}
            The density of $\text{ZE}(X,f)$ in $\overline{T'(X,f)}$ is a direct consequence of Theorem \ref{thm:Toeplitzdense}, since every regular Toeplitz subshift has zero topological entropy. 

            In order to prove that $\text{ZE}(X,f)$ is a $G_\delta$ set, we use Lemma \ref{lem:entropyuppersemi}: the sets $\{Y\in S(X,f) \mid h(Y)<\frac{1}{n}\}$ are open in the Hausdorff metric for each $n\in\N$, which follows from the upper-semicontinuity of $h\colon S(X,f)\to \mathbb{R}$. Then, we can write
                \[\text{ZE}(X,f)=\bigcap_{n\in\N} \left\{Y\in S(X,f)\mid h(Y)<\frac{1}{n} \right\},\]
            which proves that $\text{ZE}(X,f)$ is a $G_\delta$ set in $S(X,f)$. Consequently, $\text{ZE}(X,f)$ is a $G_\delta$ set in the induced topology on $\overline{T'(X,f)}$.
        \end{proof}

        \subsection{Symbolic complexity and slow entropy}

            We have just established genericity of subsystems with zero topological entropy within the locus of $\overline{T'(X,f)}$. This result combined with Corollary \ref{cor:symbolic} allows us to ask how fast the complexity of a typical subsystem grows. Recall that, when $Y$ is a subshift, we can compute the topological entropy by
                \[h(\sigma,Y)=\lim_{n\to\infty}\dfrac{\log(C(Y,n))}{n},\] 
            where $C(Y,n)$ denotes the complexity function for $Y$, i.e., $C(Y,n)$ is the total number of admissible words of length $n$ of $Y$. Theorem \ref{thm:zeroentropy} says that the growth rate of $C(Y,n)$ is, generically, subexponential, namely, $\lim_{n\to\infty}\log C(Y,n)/n=0$. In this section, we show that a typical subsystem in $\overline{T'(X,f)}$ exhibits wildly oscillating complexity. We make this idea precise throughout the section. 

            The idea of a wildly oscillating complexity function previously appeared in \cite{PavlovSchmieding}, in Proposition 5.23. We give the content of this proposition below, rephrased in our setting.
            \begin{theorem}[Proposition 5.23, \cite{PavlovSchmieding}]
              For a generic subshift $Z\in\overline{T'(\Sigma,\sigma)}$, where $\Sigma$ is an SFT, the following holds:
            \begin{itemize}
                \item For all $\gamma>1$, there is a subsequence $n_k$ so that ${\displaystyle\lim_{k\to\infty}}\frac{C(Z,{n_k})}{n_k}=\gamma$,
                \item For all $\gamma\geq 1$, there is a subsequence $n_k$ so that ${\displaystyle\lim_{k\to\infty}}\frac{\log(C(Z,n_k))}{\log n_k}=\gamma$,
                \item For all $\gamma\in (0,1)$, there is a subsequence $n_k$ so that ${\displaystyle\lim_{k\to\infty}}\frac{\log\log(C(Z,n_k))}{\log n_k}=\gamma$.
            \end{itemize}  
            \end{theorem}

            Generically, the word complexity of a subshift can grow at any linear rate $\gamma n$, at any polynomial scale $n^\gamma$, and at any stretched-exponential scale $e^{n^\gamma}$, as long as we restrict $n$ to a subsequence. We see a similar behavior for the typical subsystem in $\overline{T'(X,f)}$. Note that all these scales are subexponential.

            Our first result in this direction allows us to bound the symbolic complexity of a subshift in $S(X,f)$ by the complexity of an extension of that subshift in $S(\Sigma,\sigma)$.
            
            \begin{proposition}\label{prop:complexitybound}
                Let $\pi\colon \Sigma\to X$ be a finite-to-one, almost one-to-one factor map between an irreducible SFT $(\Sigma,\sigma)$ and an irreducible Smale space $(X,f)$. Let $Z\in \mathrm{M}(X,f)$ and $Y\in\pi_*^{-1}(Z)\subseteq S(\Sigma,\sigma)$ be subsystems of $X$ and $\Sigma$ respectively. Then, for every subshift $Z'$ in the conjugacy class of $Z$ there exist $M>0$ and $k=k(Z,Z',Y,\pi)$ such that 
                    \[M^{-1}\cdot C(Z',n-k)\leq C(Y,n)\leq M \cdot C(Z',n+k)\]
                for all $n$ sufficiently large.  
            \end{proposition}
            
            In order to prove the proposition, we first establish some auxiliary results concerning subshifts. Let $\pi:X\to Y$ be a surjective sliding block code between two-sided subshifts $X$ and $Y$, with memory $m$ and anticipation $a$, and let $\sigma$ denote the left-shift on both spaces. 

            Assume, moreover, that $\pi$ is left-closing; that is, 
                \[\pi(x)=\pi(x')\quad \text{and} \quad x_{[1,\infty)}=x'_{[1,\infty) } \,\Longrightarrow\, x=x'.\] 
                
            Equivalently, $\pi$ is injective on global stable sets. We also assume that the fibers of $\pi$ are uniformly bounded: there exists $M\geq 1$ such that $\left|\pi^{-1}(y)\right|\leq M$ for every $y\in Y$. 

            For a subshift $Z$, let $\mathcal{L}_n(Z)\coloneq \{z_{[0,n-1]} \mid z\in Z\}$ denote its language of words of length $n$, so that $C(Z,n)=\left|\mathcal{L}_n(Z)\right|$. 

            \begin{lemma}\label{lem:complexityboundsubshift}
                Let $\pi\colon X\to Y$ be a surjective sliding block code between subshifts, with memory $m$ and anticipation $a$. Suppose that $\pi$ is left-closing and that $\left|\pi^{-1}(y)\right|\leq M$ for every $y\in Y$. Then, there exists $k\in\N_0$ such that, for all sufficiently large $n$,
                    \[C(Y,n-k)\leq C(X,n)\leq M\cdot C(Y,n+k).\]
                
                    If $\pi$ is right-closing instead, the same estimate holds.
            \end{lemma}
            
            \begin{proof}
                We first prove the lower bound. For $n\geq m+a+1$, every word $w\in \mathcal{L}_n(X)$ determines, through the local rule defining $\pi$, an output word of length $n-m-a$. Thus, $\pi$ induces a map $\phi_n\colon \mathcal{L}_{n}(X)\to\mathcal{L}_{n-m-a}(Y)$. This map is surjective. Indeed, given $v\in \mathcal{L}_{n-m-a}(Y)$, choose $y\in Y$ so that $y_{[m,n-a-1]}=v$. Since $\pi$ is surjective, there exists $x\in X$ such that $\pi(x)=y$, and then $x_{[0,n-1]}$ is mapped by $\phi_n$ to $v$. Hence
                    \[C(Y,n-m-a)\leq C(X,n).\]
                Setting $k_1=m+a$, we obtain $C(Y,n-k_1)\leq C(X,n)$ for every $n>k_1$.

                We now prove the upper bound. For an integer $K\geq 0$, $n\geq 1$ and $v\in\mathcal{L}_{n+2K}(Y)$ define 
                    \[E_n(v)\coloneq\left\{w\in \mathcal{L}_n(X) \mid \exists x\in X \text{ such that }x_{[0,n-1]}=w,\, \pi(x)_{[-K,n+K-1]}=v\right\}.\]

                We claim that there exists $K\geq 0$ such that $\left|E_n(v)\right|\leq M$ for every $n\geq 1$ and every $v\in\mathcal{L}_{n+2K}(Y)$.  
                
                Suppose, toward a contradiction, that no such $K$ exists. Then, for every $K\geq 0$, there exists an integer $n_K \geq 1$ and $v_K\in\mathcal{L}_{n_K+2K}(Y)$ such that $\left|E_{n_K}(v_K)\right|\geq M+1$. Choose distinct words $w^{K,1}, \ldots, w^{K,M+1}\in E_{n_K}(v_K)$. For each $i\in\{1,\ldots, M+1\}$, choose $x^{K,i}\in X$ satisfying
                    \[x^{K,i}_{[0,n_K-1]}=w^{K,i} \quad \text{and} \quad \pi\left(x^{K,i}\right)_{[-K,n_K+K-1]}=v_K.\]

                Because the words $\{w^{K,i}\}_{i=1}^{M+1}$ are mutually distinct, for every pair $i\neq j$ we may define the last coordinate at which they differ by 
                    \[c_{K,i,j}\coloneq\max\left\{r\in \{0,\ldots, n_K-1\} \mid x^{K,i}_r\neq x^{K,j}_r\right\}.\]  
                Hence
                    \[x^{K,i}_{c_{K,i,j}}\neq x^{K,j}_{c_{K,i,j}}, \qquad x^{K,i}_{[c_{K,i,j}+1,n_K-1]}=x^{K,j}_{[c_{K,i,j}+1,n_K-1]}.\]

                Consider the distances from this last differing coordinate to the right endpoint of the word $d_{K,i,j}\coloneq n_K-1-c_{K,i,j}$. We distinguish two cases.

                \begin{enumerate}[label=\fbox{Case \arabic*}, wide=0pt]
                    \item $\limsup_{K\to\infty} d_{K,i,j}=\infty$ for at least one pair $i\neq j$.
                    
                        Passing to a subsequence, we may assume that $d_{K,i,j}\to\infty$. For this pair, define
                            \[z^{K,i}=\sigma^{c_{K,i,j}}(x^{K,i}),\qquad z^{K,j}=\sigma^{c_{K,i,j}}(x^{K,j}).\]

                        By definition, $z^{K,i}_0\neq z^{K,j}_0$ and $z^{K,i}_{[1,d_{K,i,j}]}=z^{K,j}_{[1,d_{K,i,j}]}$. Furthermore, 
                            \[\pi(x^{K,i})_{[-K,n_K+K-1]}=v_K=\pi(x^{K,j})_{[-K,n_K+K-1]}.\] 
                            
                        Since $\pi$ commutes with the shift, the images $\pi(z^{K,i})$ and $\pi(z^{K,j})$ agree on the interval $[-K-c_{K,i,j},n_K+K-1-c_{K,i,j}]=[-K-c_{K,i,j},K+d_{K,i,j}]$. Since $c_{K,i,j}\geq 0$, the left endpoint is at most $-K$. Moreover, $K+d_{K,i,j}\geq K$. Therefore, $\pi(z^{K,i})_{[-K,K]}=\pi(z^{K,j})_{[-K,K]}$.

                        By compactness of $X\times X$, after passing to a further subsequence we may assume that $z^{K,i}\to z^i$ and $z^{K,j}\to z^{j}$. Since $d_{K,i,j}\to\infty$, the right endpoint of the interval on which $z^{K,i}$ and $z^{K,j}$ agree to the right tends to $\infty$. Hence $z^i_{[1,\infty)}=z^j_{[1,\infty)}$.
                         
                        At the same time, $z^{K,i}_0\neq z^{K,j}_0$ for every $K$. Since the alphabet is finite and convergence in a subshift implies eventual agreement on every fixed coordinate, it follows that $z^i_0\neq z^j_0$. Thus $z^j\neq z^i$.
                        Moreover, the images $\pi\left(z^{K,i}\right)$ and $\pi\left(z^{K,j}\right)$ agree on $[-K,K]$, and these intervals exhaust $\Z$. Therefore, $\pi(z^i)=\pi(z^j)$. 
                        
                        We have obtained two distinct points $z^i$, $z^j$ satisfying $z^j_{[1,\infty)}=z^i_{[1,\infty)}$ and $\pi(z^i)=\pi(z^j)$, contradicting the fact that $\pi$ is left-closing.

                    \item Suppose that $\limsup_{K\to\infty}d_{K,i,j}<\infty$ for every pair $i\neq j$. 
                    
                        Since there are only finitely many pairs $i\neq j$, and each $d_{K,i,j}$ takes values in $\N_0$, we may pass to a subsequence such that, for every pair $i\neq j$, $d_{K,i,j}$ is constant along the subsequence. 
                        
                        Define $y^{K,i}=\sigma^{n_K-1}(x^{K,i})$ for $i\in\{1,\ldots,M+1\}$. By compactness of $X^{M+1}$, after passing to a further subsequence we may assume that $y^{K,i}\to y^i$ for every $i\in\{1,\ldots, M+1\}$.

                        For each $i\neq j$, $x^{K,i}_{n_K-1-d_{K,i,j}}\neq x^{K,j}_{n_K-1-d_{K,i,j}}$, and therefore $y^{K,i}_{-d_{K,i,j}}\neq y^{K,j}_{-d_{K,i,j}}$. Passing to the limit gives $y^i \neq y^j$ whenever $i\neq j$. Thus, $y^1, \ldots, y^{M+1}$ are pairwise distinct.
                        
                        We now consider their images. Since 
                        \[\pi(x^{K,i})_{[-K,n_K+K-1]}=\pi(x^{K,j})_{[-K,n_K+K-1]},\]  
                        after shifting by $n_K-1$, the images $\pi(y^{K,i})$ and $\pi(y^{K,j})$ agree on $[-K-(n_K-1),K]$. Since $n_K\geq 1$, this interval contains $[-K,K]$. Hence, 
                        \[\pi(y^{K,i})_{[-K,K]}=\pi(y^{K,j})_{[-K,K]}.\] 
                        Passing to the limit, we obtain $\pi(y^i)=\pi(y^j)$ for every $i,j$.

                        Thus, we have found $M+1$ pairwise distinct points $y^1, \ldots, y^{M+1}$ that belong to the same fiber of $\pi$. This contradicts the assumption that $\left|\pi^{-1}(y)\right|\leq M$ for every $y\in Y$.
                \end{enumerate}

                Both cases yield a contradiction. Therefore, there exists $K\geq 0$ such that, for every $n\geq 1$ and $v\in\mathcal{L}_{n+2K}(Y)$, we have $\left|E_n(v)\right|\leq M$. 
                
                Fix such a $K$. For every word $w\in\mathcal{L}_n(X)$, choose $x\in X$ such that $x_{[0,n-1]}=w$. Then, setting $v=\pi(x)_{[-K,n+K-1]}$, we have $w\in E_n(v)$. Therefore, 
                    \[ \mathcal{L}_n(X)\subseteq \bigcup_{v\in \mathcal{L}_{n+2K}(Y)}E_n(v).\]
                The sets $E_n(v)$ need not be disjoint, but this inclusion is sufficient to get
                    \[C(X,n)\leq \sum_{v\in\mathcal{L}_{n+2K}(Y)}\left|E_n(v)\right|.\]
                Using the uniform bound $\left|E_n(v)\right|\leq M$, we obtain 
                    \[C(X,n)\leq M\cdot C(Y,n+2K).\]

                Setting $k_2=2K$, we obtain $C(X,n)\leq M\cdot C(Y,n+k_2)$ for every $n\in\N$.

                Finally, let $k=\max\{k_1,k_2\}$. Since the complexity function is nondecreasing,
                    \[C(Y,n-k)\leq C(Y,n-k_1)\leq C(X,n) \leq M\cdot C(Y,n+k_2)\leq M\cdot C(Y,n+k)\]
                for all sufficiently large $n$. This proves the lemma when $\pi$ left-closing. The argument for the right-closing case is analogous.
            \end{proof}

            Observe that, if we consider a map $\pi:X\to Y$ between Smale spaces that is $s$-resolving, then its restriction to a subshift $Z\subseteq X$ is left-closing. Analogously, a $u$-resolving map has a right-closing restriction.

            \begin{proof}[Proof of Proposition \ref{prop:complexitybound}]
            In the original setting we have the map $\pi\colon \Sigma\to X$ between irreducible Smale spaces, which is almost one-to-one. By Theorem 1.1 in \cite{PutnamLift} there exist irreducible Smale spaces $(\tilde{\Sigma},\tilde{\sigma})$ and $(\tilde{X},\tilde{f})$, $u$-resolving factor maps $\alpha\colon \tilde{\Sigma}\to\Sigma$, $\beta\colon \tilde{X}\to X$, and an s-resolving factor map $\tilde{\pi}:\tilde{\Sigma}\to\tilde{X}$ such that the following diagram commutes:

            \begin{center}
                \begin{tikzcd}
                    \tilde{\Sigma} \arrow[rr, "\tilde{\pi}"] \arrow[dd, "\alpha"'] &  & \tilde{X} \arrow[dd, "\beta"]  \\
                                                                                &  &                                   \\
                    \Sigma \arrow[rr, "\pi"]                                       &  & X                              \\                                                           
                \end{tikzcd}
            \end{center}

            Let $Z\in \mathrm{M}(X,f)$ and $Y\in\pi_*^{-1}(Z)\subseteq T(\Sigma,\sigma)$. According to Corollary \ref{cor:symbolic}, the subsystem $Z$ is conjugate to a subshift $Z'$ via a homeomorphism, which we call $h\colon Z\to Z'$. Since $\beta$ is continuous and $u$-resolving, it is finite-to-one (\cite{PutnamMemoir}, Theorem 2.5.3) with a uniform bound $M_\beta$ for the fiber size of $\beta$. The next lemma shows that $\beta^{-1}(Z)\subseteq \tilde{X}$ is totally disconnected and expansive. Since it is a compact metric space, total disconnectedness is equivalent to zero-dimensionality. Therefore, the compact $\tilde{f}$-invariant subset $\tilde{Z}_Y\coloneq\tilde{\pi}\big(\alpha^{-1}(Y)\big)\subseteq\beta^{-1}(Z)$ is conjugate to a subshift $B$ via a homeomorphism $g\colon\tilde{Z}_Y\to B$.

            \begin{lemma}
                Let $\pi\colon Y\to Z$ be a continuous, finite-to-one surjection between compact Hausdorff spaces. If $Z$ is totally disconnected, then so is $Y$.
            \end{lemma}

            \begin{proof}
                Let $y\in Y$ and $C_y$ be the connected component of $y$ in $Y$. Then $\pi(C_y)\ni \pi(y)$ is a connected set of $Z$, which is totally disconnected. Thus, $\pi(C_y)=\{\pi(y)\}$. Also, $|C_y|\leq \left|\pi^{-1}(\pi(C_y))\right|<\infty$, so $C_y$ is finite and contains $y$. Since $Y$ is Hausdorff, it must be $C_y=\{y\}$.  
            \end{proof}

            Similarly, $\alpha^{-1}(Y)$ is totally disconnected, and, consequently, it is conjugate to a subshift $A$ via a homeomorphism $k:\alpha^{-1}(Y)\to A$. The following diagram commutes:
            \begin{center}
                \begin{tikzcd}
                                                                                            & A \arrow[rr] &                                                    & B \arrow[dd] \\
                \alpha^{-1}(Y) \arrow[rr, "\tilde{\pi}"] \arrow[dd, "\alpha"'] \arrow[ru, "k"] &              & \tilde{Z}_Y \arrow[dd, "\beta"'] \arrow[ru, "g"] &              \\
                                                                                            &              &                                                    & Z'           \\
                Y \arrow[rr, "\pi"]                                                            &              & Z \arrow[ru, "h"]                                  &             
                \end{tikzcd}
            \end{center}
            It is easy to check that $\beta\left(\tilde{Z}_Y\right)=Z$, so $g\circ\tilde{\pi}\circ k^{-1}\colon A\to B$ is left-closing and $\alpha\circ k^{-1}\colon A\to Y$, $h\circ\beta\circ g^{-1}\colon B\to Z'$ are both right-closing factor maps. These maps are also finite-to-one, with fiber sizes at most $M_{\tilde{\pi}}$, $M_\alpha$ and $M_\beta$ respectively. By the Curtis-Hedlund-Lyndon theorem, they are sliding block codes as well, since they are continuous and shift-commuting.
            By Lemma \ref{lem:complexityboundsubshift}, there exist $k_1, k_2, k_3\in \Z$ such that for sufficiently large $n$
                \begin{align*}
                    C(Z',n-k_1)\leq &C(B,n)\leq M_\beta\cdot C(Z',n+k_1),\\
                    C(B,n-k_2)\leq & C(A,n)\leq M_{\tilde{\pi}}\cdot C(B,n+k_2),\\
                    C(Y,n-k_3)\leq &C(A,n)\leq M_\alpha\cdot C(Y,n+k_3).
                \end{align*} 
            Setting $k=k_1+k_2+k_3$, we combine the inequalities above to obtain
                \begin{align*}
                    C(Z',n-k)\leq C(B,n-k_2-k_3) \leq C(A,n-k_3)\leq M_\alpha\cdot C(Y,n),\\
                    C(Y,n)\leq C(A,n+k_3)\leq M_{\tilde{\pi}}\cdot C(B,n+k_3+k_2) \leq M_{\tilde{\pi}}M_\beta \cdot C(Z',n+k)
                \end{align*}
            for every $n$ sufficiently large. In conclusion,
                \[M_\alpha^{-1}\cdot C(Z',n-k)\leq C(Y,n)\leq M_{\tilde{\pi}}M_\beta\cdot C(Z',n+k) \quad \forall n\geq k.\]
            Defining $M=\max\{M_\alpha, M_{\tilde{\pi}}M_\beta\}$, we obtain the inequality in the proposition. Note that the argument holds for every subshift $Z'$ in the conjugacy class of $Z$. The proposition is proved. 
            \end{proof}

            Before stating the main result of the subsection, we provide the definition of topological slow entropy. For a more extensive treatment of this topic, see \cite{SlowEntropy}.

            \begin{definition} 
                A scale $\{a_\chi\}$ is a family of increasing functions $a_\chi\colon\N\to (0,\infty)$ indexed by $\chi\in\mathbb{R}_{\geq 0}$, such that if $\chi<\chi'$ then $a_{\chi}=o(a_{\chi'})$.
            \end{definition}

            \begin{definition} 
                Let $(X,f)$ be a topological dynamical system. Let $S_Z(n,\eps)$ denote the maximal cardinality of an $(n,\eps)$-separated set of a compact $f$-invariant subset $Z\subseteq X$. The slow topological entropy of $(Z,f)$ for the scale $a_\chi$ is
                \[h_{\mathrm{top},a_\chi}(f,Z)=\lim_{\eps\to 0^+}\sup\left\{\chi\geq 0 \mid \limsup_{n\to\infty}\dfrac{S_Z(n,\eps)}{a_\chi(n)}>0\right\}.\]
            \end{definition}

            In the definition of $h_{\mathrm{top},a_\chi}(f,Z)$ we can omit the limit as $\eps\to 0^+$ whenever the system is expansive. In \cite{Pujals}, the authors define a generalized entropy that extends the definition of slow entropy and prove that there exists $\eps\in (0,c_f)$, where $c_f$ is the expansivity constant of $f$, such that 
                \[h_{\mathrm{top},a_\chi}(f,Z)=\sup\left\{\chi\geq 0 \mid \limsup_{n\to\infty}\dfrac{S_Z(n,\eps)}{a_{\chi}(n)}>0\right\}.\]
            In the same article, it is proved that their generalized entropy is invariant under topological conjugacy.
            
            Observe that, in the symbolic case covered in \cite{PavlovSchmieding}, linear growth estimates for the complexity are provided. However, the family of linear functions $\{\chi n\}$ is not a scale, since they do not satisfy $a_{\chi}=o(a_{\chi'})$ for $\chi<\chi'$. In this case, we state the result in terms of linear complexity. 

            \begin{definition}
                Let $C(Y,n)$ denote the symbolic complexity of a subshift $(Y,\sigma)$ over a finite alphabet. We say that $Y$ has linear complexity if 
                    \[\limsup_{n\to\infty}\dfrac{C(Y,n)}{n}<\infty.\]
            \end{definition}

            We need one more auxiliary result before proving the main theorem of the section.
            
            \begin{proposition}\label{prop:coplexitysubsequence}
                Let $(X,f)$ be an irreducible Smale space. There exists a dense subset of $\overline{T'(X,f)}$ consisting of minimal subsystems $Z$ such that, for every subshift $Z'$ conjugate to $Z$,
                \begin{enumerate}[wide=0pt, font=\upshape]
                    \item for every $\gamma\geq 1$, 
                        \[\dfrac{\log(C(Z',n_j))}{\log n_j}\to \gamma\]
                    along some subsequence $n_j\to\infty$;
                    \item for every $\gamma\in (0,1)$
                        \[\dfrac{\log\log C(Z',m_j)}{\log m_j}\to \gamma\]
                    along some subsequence $m_j\to\infty$.
                \end{enumerate}
                Moreover, 
                    \[\limsup_{n\to\infty}\dfrac{C(Z',n)}{n}=\infty.\]
            \end{proposition}

            \begin{proof}
                Recall Proposition 5.23 in \cite{PavlovSchmieding}. Define $\mathcal{D}$ as the intersection of $\mathrm{M}(\Sigma,\sigma)$ and the residual set in $\overline{T'(\Sigma,\sigma)}$ where the proposition holds. Then $\mathcal{D}$ is also residual in $\overline{T'(\Sigma,\sigma)}$, and so $\pi_*(\mathcal{D})$ is dense in $\overline{T'(X,f)}$ by Lemma \ref{lem:transitivesurj}. 

                Let $Z\in\pi_*(\mathcal{D})$, which is minimal and expansive. Let $Z'$ be a subshift conjugate to $Z$, which we know exists by Corollary \ref{cor:symbolic}. Then, by Proposition \ref{prop:complexitybound}, there are $Y\in \pi_*^{-1}(Z)\cap \mathcal{D}$, $M>0$ and $k(Z,Z',Y,\pi)\in \N$ such that
                    \begin{equation}\label{eq:complexitybound}
                        M^{-1}\cdot C(Z',n-k)\leq C(Y,n)\leq M\cdot C(Z',n+k)
                    \end{equation}
                for all sufficiently large $n$. 

                To establish (1), let $n_j\to\infty$ be the subsequence such that $\dfrac{\log C(Y,n_j)}{\log n_j}\to\gamma$. Using inequality (\ref{eq:complexitybound}) and submultiplicativity of the complexity function,
                    \begin{align*} 
                        \dfrac{\log C(Z',n_j-k)-\log M}{\log (n_j-k)}&\leq \dfrac{\log C(Y,n_j)}{\log (n_j-k)} \\&\leq \dfrac{\log C(Z',n_j-k)+\log C(Z',2k)+\log M}{\log (n_j-k)}.
                    \end{align*}
                Note that $\lim_{j\to\infty}\tfrac{\log (n_j+c)}{\log n_j}=1$ for any constant $c\in\mathbb{R}$. Hence
                    \[\limsup_{j\to\infty}\dfrac{\log C(Z',n_j-k)}{\log (n_j-k)}\leq \lim_{j\to\infty} \dfrac{\log C(Y,n_j)}{\log n_j}\cdot \dfrac{\log n_j}{\log(n_j-k)}=\gamma.\]

                Similarly, since $C(Z',2k)$ is a constant,
                    \[\liminf_{j\to\infty}\dfrac{\log C(Z',n_j-k)}{\log(n_j-k)}\geq \lim_{j\to\infty}\dfrac{\log C(Y,n_j)}{\log n_j}\cdot \dfrac{\log n_j}{\log(n_j-k)}=\gamma.\]
                Thus, convergence holds along the subsequence $n_j-k\to\infty$.
                
                The proof of (2) is analogous. Let $m_j$ be the sequence such that $\frac{\log\log C(Y,m_j)}{\log m_j}\to \gamma$. Submultiplicativity of $C(Y,n)$ yields
                    \[C(Y,k)^{-1}C(Y,m_j)\leq C(Y,m_j-k)\leq C(Y,m_j+k)\leq C(Y,k)C(Y,m_j.)\]
                Hence, setting $B_j\coloneq \log C(Y,m_j)$, we have
                    \[B_j-\log C(Y,k) \leq \log C(Y,m_j-k)\leq \log C(Y,m_j+k) \leq B_j+\log C(Y,k).\]
                Therefore, using inequality (\ref{eq:complexitybound}),
                    \[B_j-\log C(Y,k)-\log M\leq \log C(Z',m_j)\leq B_j +\log C(Y,k)+\log M.\]
                Define $K\coloneq \log C(Y,k)+\log M$. Since $\gamma>0$, $\frac{\log B_j}{\log m_j}\to \gamma$ implies that $B_j\to\infty$. Hence, for large $j$
                    \[\log(B_j-K)\leq \log\log C(Z',m_j)\leq \log (B_j+K).\]
                Now $\log(B_j-K)-\log B_j\to 0$, and similarly $\log B_j-\log(B_j+K)\to 0$. Thus,
                    \[\dfrac{\log(B_j-K)}{\log m_j}\to \gamma, \qquad \dfrac{\log (B_j+K)}{\log m_j}\to\gamma.\]
                The squeeze theorem gives the desired convergence . Thus, convergence holds along the subsequence $m_j\to\infty$.

                For the linear growth rate statement, fix $\gamma\geq 1$ and choose a corresponding subsequence $n_j\to\infty$ such that $\frac{C(Y,n_j)}{n_j}\to\gamma$. We use inequality (\ref{eq:complexitybound}) again, so
                    \[\dfrac{C(Z',n_j+k)}{n_j+k}\geq \dfrac{1}{M}\cdot \dfrac{C(Y,n_j)}{n_j}\cdot \dfrac{n_j}{n_j+k}.\]
                Therefore,
                    \[\limsup_{n\to\infty}\dfrac{C(Z',n)}{n}\geq\limsup_{j\to\infty}\frac{C(Z',n_j)}{n_j}\geq \dfrac{\gamma}{M}.\]
                Since $\gamma\geq 1$ is arbitrary, the statement follows.
            \end{proof}

            Observe that the multiplicative constant $M$ in inequality (\ref{eq:complexitybound}) prevents us from extracting a subsequence $n_j$ so that $C(Z',n_j+k)$ has linear growth $\gamma n_j$. 

            We can finally state the main result of the section.

            \begin{theorem}
                Let $(X,f)$ be an irreducible Smale space. Then there is a residual set $\mathcal{G}\subseteq \overline{T'(X,f)}$ such that every $Z\in\mathcal{G}$ satisfies 
                    \[h_{\mathrm{top},n^\chi}(f,Z)=\infty\]
                for the polynomial scale $a_\chi(n)=n^\chi$, and
                    \[h_{\mathrm{top},e^{n^\chi}}(f,Z)=1\]
                for the stretched-exponential scale $a_\chi(n)=e^{n^\chi}$.

                Moreover, for every sufficiently small $\eps>0$, 
                    \[\limsup_{n\to\infty}\dfrac{S_Z(n,\eps)}{n}=\infty.\]
            \end{theorem}

            \begin{proof}
                Fix $\eps>0$ sufficiently small so that the Correa-Pujals expansive-scale result (Proposition B.4 in \cite{Pujals}) applies to $(X,f)$. Then every subsystem $Z\subseteq X$ satisfies
                    \[h_{\mathrm{top},a_\chi}(f,Z)=\sup\left\{ \chi \geq 0 \mid \limsup_{n\to\infty} \dfrac{S_Z(n,\eps)}{a_\chi(n)}>0\right\}.\]
                Let $d_n(x,y)\coloneq\max_{0\leq j<n}d(f^j(x),f^j(y))$ denote the Bowen metric. We say that a set $E$ is $(n,\eps)$-separated if $d_n(x,y)>\eps$ for every pair of distinct $x,y\in E$. Recall that $S_Z(n,\eps)\coloneq\max\{|E| \mid E\subseteq Z \text{ is } (n,\eps)-\mathrm{separated}\}$. 
                
                For $n, q\in\N$ define
                    \begin{align*} 
                        U(n,q)&\coloneq \{Z\in \overline{T'(X,f)} \mid S_Z(n,\eps)\geq q\}\\
                        &   =\{Z\in\overline{T'(X,f)} \mid \exists x_1, \ldots, x_q\in Z \text{ so that }d_n(x_i,x_j)>\eps \text{ for all }i\neq j\}.
                    \end{align*}
                Equivalently, $Z\in U(n,q)$ if there are $x_1, \ldots, x_q\in Z$ with $d_n(x_i,x_j)>\eps$ whenever $i\neq j$. These sets are open: if $x_1,\ldots, x_q$ are witnesses of the latter property, then there exists $\eta>0$ so that
                    \[0<2\eta<\min_{i\neq j}\{d_n(x_i,x_j)-\eps\}.\]
                Since the map $(x,y)\mapsto d_n(x,y)$ is continuous on the compact space $X\times X$, there exists $\delta>0$ such that $d(x,x')<\delta \Longrightarrow d_n(x,x')<\eta$.

                Now suppose that $Z'\in S(X,f)$ satisfies $d_{\mathcal{H}}(Z,Z')<\delta$. For each $i$, choose $x_i'\in Z'$ such that $d(x_i,x_i')<\delta$. Then $d_n(x_i,x_i')<\eta$.
                Hence, for $i\neq j$, the triangle inequality for $d_n$ gives
                    \begin{align*}
                        d_n(x_i',x_j')&\geq d_n(x_i,x_j)-d_n(x_i,x_i')-d_n(x_j,x_j')\\
                                      &> (\eps+2\eta)-\eta-\eta=\eps.
                    \end{align*}
                Thus, $x_1',\ldots, x_q'\in Z'$ form an $(n,\eps)$-separated subset of $Z'$, and therefore $Z'\in U(n,q)$. 
                Hence $U(n,q)$ is open in the Hausdorff metric.

                We now prove that $h_{\mathrm{top},n^\chi}(f,Z)=\infty$ holds on a dense $G_\delta$ set of $\overline{T'(X,f)}$. 

                For any positive sequence $a(n)$, define
                    \[\mathcal{G}(a)=\bigcap_{q\in\N}\bigcap_{N\in\N}\bigcup_{n\geq N}U(n,\lfloor qa(n)\rfloor +1),\]
                which is a $G_\delta$ set since $U(n,q)$ is open for every $n,q\in\N$.
                Observe that 
                    \[Z\in \mathcal{G}(a)\Longleftrightarrow \limsup_{n\to\infty}\dfrac{S_Z(n,\eps)}{a(n)}=\infty.\]

                For the statement about the polynomial scale $a_{\chi}(n)=n^\chi$, we characterize $h_{\mathrm{top},n^\chi}(f,Z)=\infty$ by 
                    \[\mathcal{G}_{\text{poly}}=\bigcap_{\chi\in\mathbb{Q}_{>0}}\mathcal{G}(n^\chi)\]
                which is a $G_\delta$ set. 
                Similarly, we characterize $h_{\mathrm{top},e^{n^\chi}}(f,Z)\geq 1$ by 
                    \[\mathcal{G}_{\text{str}}^-\coloneq\bigcap_{\chi\in\mathbb{Q}\cap(0,1)}\mathcal{G}\left(e^{n^\chi}\right)\]
                and $h_{\mathrm{top},e^{n^\chi}}(f,Z)=1$ by $\mathcal{G}_{\text{str}}\coloneq\mathcal{G}_{\text{str}}^-\cap \text{ZE}(X,f)$, which are also $G_\delta$ sets.
                
                For the density, let $\mathcal{R}\coloneq\pi_*(\mathcal{D})$, where $D$ is the set defined in the proof of Proposition \ref{prop:coplexitysubsequence}. We prove directly that $\mathcal{R}\subseteq \mathcal{G}(n^\chi)$ for every $\chi>0$ and $\mathcal{R}\subseteq \mathcal{G}(e^{n^\chi})$ for every $0<\chi<1$. 

                Fix $Z\in\mathcal{R}$ and $\chi>0$. By Proposition \ref{prop:coplexitysubsequence}, $Z$ has a symbolic representative $Z'$ with oscillating polynomial complexity. 
                Hence $h_{\mathrm{top},n^\chi}(\sigma,Z')=\infty$, and by conjugacy invariance of the slow entropy
                    \[h_{\mathrm{top},n^\chi}(f,Z)=\infty.\]
                Since slow entropy for $f\big|_{Z}$ is computed at the fixed scale $\eps$, for every $\gamma>\chi$, 
                    \[\limsup_{n\to\infty}\dfrac{S_Z(n,\eps)}{n^\gamma}>0.\]
                Fix such a $\gamma$. There are $c>0$ and $n_j\to\infty$ such that $S_Z(n_j,\eps)\geq cn_j^\gamma$ for every $j$. Therefore 
                    \[\dfrac{S_Z(n_j,\eps)}{n_j^\chi}\geq cn_j^{\gamma-\chi}\to\infty.\]
                Thus    
                    \[\limsup_{n\to\infty}\dfrac{S_Z(n,\eps)}{n^\chi}=\infty,\]
                so $Z\in \mathcal{G}(n^\chi)$. 

                Since $\mathcal{R}$ is dense in $\overline{T'(X,f)}$, $\mathcal{G}(n^\chi)$ is dense for every $\chi>0$. By the Baire Category Theorem, $G_{\text{poly}}$ is residual in $\overline{T'(X,f)}$. 

                For the stretched-exponential growth, fix $Z\in\mathcal{R}$ and $0<\chi<1$. The stretched-exponential oscillations of the symbolic representative $Z'$ give $h_{\mathrm{top},e^{n^\chi}}(\sigma,Z')\geq 1$, and hence, by conjugacy invariance, $h_{\mathrm{top},e^{n^\chi}}(f,Z)\geq 1$. Choose $\gamma\in (\chi,1)$. Again, using that $\eps$ computes slow entropy,
                    \[\limsup_{n\to\infty}\dfrac{S_Z(n,\eps)}{e^{n^\gamma}}>0.\]
                Thus, there are $c>0$ and $n_j\to\infty$ such that $S_Z(n_j,\eps)\geq ce^{n_j^\gamma}$. Consequently,
                    \[\dfrac{S_Z(n_j,\eps)}{e^{n_j^\chi}}\geq c \exp(n_j^\gamma-n_j^\chi)\to\infty,\]
                because $\gamma>\chi$. Therefore, 
                    \[\limsup_{n\to\infty}\dfrac{S_Z(n,\eps)}{e^{n^\chi}}=\infty,\]
                so $Z\in \mathcal{G}\left(e^{n^\chi}\right)$. We conclude, as in the polynomial-scale case, that every $\mathcal{G}\left(e^{n^\chi}\right)$ is dense. The Baire Category Theorem allows us to conclude that $\mathcal{G}_{\text{str}}^-$ is residual in $\overline{T'(X,f)}$.
                
                The desired properties hold in the residual set 
                    \[\mathcal{G}_{\text{poly}}\cap \mathcal{G}_{\text{str}}^-\cap \text{ZE}(X,f).\]

                The statement about $\limsup_{n\to\infty}\dfrac{S_Z(n,\eps)}{n}=\infty$ follows immediately from the infinite topological slow entropy at the polynomial scale.
            \end{proof}
                
    \section{The space of infinite totally transitive systems \texorpdfstring{$TT'(X,f)$}{TT'(X,f)}}\label{sec:totallytransitive}
        We now focus our study on the more restrictive locus of totally transitive subsystems of $(X,f)$. Such a subsystem is finite if and only if it is a singleton. Thus, we consider the set of infinite totally transitive subsystems of $S(X,f)$
            \[TT'(X,f)\coloneq \{Y\in S(X,f) \mid |Y|=\infty,\, Y \text{ is totally transitive}\}.\]

        It is worth noting that if the Smale space $(X,f)$ is transitive and not mixing, its nontrivial spectral decomposition with $p>1$ implies that $TT'(X,f)=\emptyset$. Therefore, we restrict this section to the case when $(X,f)$ is topologically mixing. In this setting, any SFT $(\Sigma,\sigma)$ that extends $(X,f)$ via $\pi\colon \Sigma\to X$ arising from a Markov partition is topologically mixing as well (Theorem 4.3.5 in \cite{AokiHiraide}).

        In analogy with the case of transitive subsystems, a point $y\in Y$ which is a transitive point for $f^{n}$ for every $n$ is called a totally transitive point. Every totally transitive system has a dense $G_\delta$ set of totally transitive points. 
        This fact, together with an argument identical to the one presented in Lemma \ref{lem:transitivesurj} implies the following result.

        \begin{lemma}
            Let $(X,f)$ be a mixing Smale space. The image of $\overline{TT'(\Sigma,\sigma)}$ under the induced map $\pi_*$ is $\overline{TT'(X,f)}$.
        \end{lemma}

        \begin{proof}
            Following the proof of Lemma \ref{lem:transitivesurj}, we get that $\pi_*(TT'(\Sigma,\sigma))=TT'(X,f)$. Since the map $\pi_*\colon S(\Sigma,\sigma)\to S(X,f)$ is continuous, it holds that 
            \[\pi_*(\overline{TT'(\Sigma,\sigma)})=\overline{TT'(X,f)}.\]
        \end{proof}

        \subsection{Genericity of topological properties}
        To apply the Baire Category Theorem, we state the results in terms of the complete set $\overline{TT'(X,f)}$. In this space, since $(Y,f)$ is totally transitive if and only if $(Y,f^n)$ is transitive for every $n\in\N$, the same argument as in Lemma \ref{lem:TdenseGdelta} shows that the countable intersection
            \[TT'(X,f)=\left(\bigcap_{n\in\N}T(X,f^n)\right)\cap S'(X,f)\]
        is a dense $G_\delta$ subset of $\overline{TT'(X,f)}$.
        Moreover, the same argument as in Lemma \ref{lem:Tperfect}, combined with Lemmas 6.2 and 6.3 in \cite{PavlovSchmieding} yield that $\overline{TT'(X,f)}$ is a perfect subset of $S(X,f)$. 

        In the following theorem we use Sturmian subshifts. For a definition of this type of subshift, see Section 13.7 of \cite{LindMarcus}.

        \begin{theorem}
            Let $(X,f)$ be a mixing Smale space. The family of subsystems that are minimal, uniquely ergodic, and have zero topological entropy is generic in $\overline{TT'(X,f)}$. 
        \end{theorem}

        \begin{proof}
            In the space $S(\Sigma,\sigma)$, the conjugacy class of any Sturmian subshift of $\Sigma$ is dense in $\overline{TT'(\Sigma,\sigma)}$ (Theorem 6.4, \cite{PavlovSchmieding}). Therefore, the previous lemma proves that the image of this conjugacy class under $\pi_*$ is dense in $\overline{TT'(X,f)}$. Recall that, since $\pi$ is finite-to-one, minimality and zero topological entropy are preserved under $\pi_*$. Furthermore, by Proposition 3.11 in \cite{DGS}, unique ergodicity is also preserved by $\pi_*$. We conclude that these properties hold on a dense set of $\overline{TT'(X,f)}$. 
            The proof of the fact that $\text{M}(X,f)\cap\text{ZE}(X,f)\cap \text{UE}(X,f)\cap \overline{TT'(X,f)}$ is a $G_\delta$ set of $\overline{TT'(X,f)}$ is identical to the proofs of Theorem \ref{thm:MinimalgenericinT}, Theorem \ref{thm:zeroentropy} and Theorem \ref{thm:UEgenericinT}.
        \end{proof}

        Within the space $\overline{TT'(X,f)}$, the family of subsystems conjugate to regular Toeplitz subshifts is no longer dense, since a regular Toeplitz subshift is never totally transitive. We now describe the family of subsystems that is dense in $\overline{TT'(X,f)}$.

        \begin{proposition}\label{prop:sturmiandense}
            Let $(X,f)$ be a mixing Smale space, $(\Sigma,\sigma)$ a mixing SFT and $\pi\colon\Sigma\to X$ a finite-to-one factor map. Let $S$ be a Sturmian subshift of $\Sigma$ and $\mathcal{C}$ its conjugacy class. Then $\mathcal{C}\cap S(X,f)$ is a dense subset of $\overline{TT'(X,f)}$.
        \end{proposition}

        \begin{proof}
            Let $S$ be a Sturmian subshift of $(\Sigma,\sigma)$, which exists by Krieger's embedding theorem. Denote by $\mathcal{C}(S)$ its conjugacy class in $\Sigma$, which is dense in $\overline{TT'(\Sigma,\sigma)}$ by Theorem 6.4 in \cite{PavlovSchmieding}. Then $\pi_*(\mathcal{C}(S))$ is dense in $\overline{TT'(X,f)}$. For every subsystem $Y\in\pi_*(\mathcal{C}(S))$ there exist $Z\in\pi_*^{-1}(Y)$ and a conjugacy $\phi_Z\colon S\to Z$, yielding a finite-to-one factor map $\pi\circ\phi_Z\colon S \to Y$ between topological dynamical systems, so $Y$ is aperiodic and minimal. If $Y$ is not conjugate to $S$, then by Proposition 11 in \cite{Durand}, $Y$ is a factor of the maximal equicontinuous factor of $(S,\sigma)$, which is an irrational rotation $(S^1,R_\alpha)$. However, $Y$ inherits the expansivity of $(X,f)$ and therefore this case is impossible. It must be that $Y$ is conjugate to $S$, so the conjugacy class of $S$ is also dense in $\overline{TT'(X,f)}$. 
        \end{proof}

        \subsection{Weakly mixing measures for a mixing subshift of finite type}
                In this subsection we prove that weak mixing measures for a mixing SFT are generic in the weak-\textasteriskcentered\ topology. This result is needed in Subsection \ref{subsec:WMsubsystem} in order to obtain a $G_\delta$ subset of $S(X,f)$ as the preimage of this residual set under a suitable continuous map.
                
                Let $(\Sigma,\sigma)$ be a mixing SFT over a finite alphabet. A measure $\mu\in\mathcal{M}_\sigma(\Sigma)$, the space of $\sigma$-invariant Borel probability measures, is weakly mixing if and only if for every $f\in L^2(\mu)$
                    \[\lim_{n\to\infty}\dfrac{1}{n}\sum_{j=1}^n \left|\int (f\circ \sigma^j)f\,d\mu-\left(\int f\,d\mu\right)^2\right|^2=0.\]
                A proof of this characterization can be found in Theorem 1.23 (ii) in \cite{Walters}. We can replace $f\in L^2(\mu)$ with $f\in C(\Sigma,\mathbb{R})$ because of the density of the latter in the former space. Note that this equivalence holds for every $(X,f)$ if $X$ is a compact metric space and $f$ is continuous.

                We use this equivalent definition of a weakly mixing measure to prove that weak mixing is a generic property in $\mathcal{M}_\sigma(\Sigma)$. This space is endowed with the weak-\textasteriskcentered\ topology. 

                In \cite{Sigmundconnectedness}, Sigmund writes that weakly mixing invariant measures form a dense \(G_\delta\) for a mixing shift of finite type, referring to earlier results for Axiom A diffeomorphisms. Since the passage from those results to the symbolic statement is not made explicit there, we give a direct proof below.

                \begin{lemma}\label{lem:weakmixingSFTGdelta}
                    Let $(\Sigma,\sigma)$ be a mixing SFT. The set of its weakly mixing measures is a $G_\delta$ set in $\mathcal{M}_\sigma(\Sigma)$.
                \end{lemma}

                \begin{proof}
                    Let $\{f_k\}_{k\in\N}$ be a dense sequence of functions in the separable space $C(\Sigma,\mathbb{R})$. For each $k\in\N$ define the quantities
                        \[A_n(k,\mu)\coloneq \dfrac{1}{n}\sum_{j=1}^n \left|\int (f_k\circ\sigma^ j)f_k\,d\mu-\left(\int f_k\,d\mu\right)^2\right|^2, \quad A(k,\mu)\coloneq \liminf_{n\to\infty}A_n(k,\mu).\]
                    The limit $\lim_{n\to\infty}A_n(k,\mu)$ exists for every $\sigma$-invariant measure $\mu\in\mathcal{M}_\sigma(\Sigma)$, as follows from von Neumann's ergodic theorem applied to the system $(\Sigma\times\Sigma,\sigma\times\sigma,\mu\times\mu)$. A proof of this fact can be found in \cite{ParthasarathyMixing}. Therefore, denoting by $\text{WMM}(\Sigma,\sigma)\subseteq\mathcal{M}_\sigma(\Sigma)$ the set of weakly mixing measures of $(\Sigma,\sigma)$, it holds that
                        \begin{align*}
                            \text{WMM}(\Sigma,\sigma)&=\{\mu\in\mathcal{M}_\sigma(\Sigma) \mid A(k,\mu)=0\,\forall k\in\mathbb{N}\}\\
                            &=\bigcap_{k\in\N}\bigcap_{l\in\N}\bigcap_{m\in\N}\bigcup_{n\geq m}\{\mu\in\mathcal{M}_\sigma(\Sigma) \mid A_n(k,\mu)<1/l\}.
                        \end{align*}
                    The set $\{\mu\in\mathcal{M}_\sigma(\Sigma) \mid A_n(k,\mu)<1/l\}$ is open in the weak-\textasteriskcentered\ topology whenever $k$, $n$ and $l$ are fixed since $A_n(k,\mu)$ is weak-\textasteriskcentered\ continuous. The result follows. 
                \end{proof}

                Observe that this proof is valid for any dynamical system $(X,f)$ with $X$ compact metric and $f$ continuous.
                Periodic measures are dense in $\mathcal{M}_\sigma(\Sigma)$ if the subshift $(\Sigma,\sigma)$ has the specification property (Theorem 1 in \cite{Sigmund}), which is true when $\Sigma$ is a mixing SFT. In order to prove density of $\text{WMM}(\Sigma,\sigma)$ in $\mathcal{M}_\sigma(\Sigma)$, we prove that the strongly mixing measures approximate periodic measures arbitrarily well. We do so by means of thermodynamic formalism for SFTs.

                We first adapt a result in \cite{Coudene} to our setting. We present the proof here since it is brief.

                \begin{lemma}\label{lem:urysohn}
                    Let $X$ be a compact metric space, $\sigma\colon X\to X$ continuous, $f:X\to\mathbb{R}$ a bounded continuous function, $K\subseteq X$ a compact subset of $X$ such that $\sigma(K)=K$, $V\supseteq K$ a neighborhood of $K$ and $c'<c$ such that
                    \begin{enumerate}[font=\upshape]
                        \item $f(x)=c$ $\forall x\in K$,
                        \item $f(x)\leq c$ $\forall x\in X$,
                        \item $f(x)\leq c'$ $\forall x\in V^c$.
                    \end{enumerate}
                    Then, for any equilibrium measure $\mu$ associated with $f$ 
                        \[\mu(X\setminus V)\leq \dfrac{h_\mu(\sigma)}{c-c'}.\]
                \end{lemma}

                \begin{proof}
                    Consider the restricted map $\sigma_K\coloneq \sigma\big|_{K}\colon K\to K\in C(K,K)$. By the Krylov--Bogoliubov theorem, there is a $\sigma_K$-invariant measure $\mu_K$ supported on $K$. We extend $\mu_K$ to a $\sigma$-invariant measure on $X$ $\tilde{\mu}_K$ by $\tilde{\mu}_K(A)\coloneq \mu_K(A\cap K)$. Since $X$ is compact metric and $f$ is continuous, the variational principle for the pressure holds (\cite{Bowen}).
                    Therefore, if $\mu$ is an equilibrium measure for the potential $f$, then 
                        \[h_\mu(\sigma)+\int f\,d\mu\geq h_{\tilde{\mu}_K}(\sigma)+\int f\,d\tilde{\mu}_K\geq \int f\,d\tilde{\mu}_K=c.\]
                    The upper bounds that $f$ satisfies give the estimate
                        \[\int f\,d\mu=\int_V f\,d\mu+\int_{X\setminus V}f\,d\mu\leq c\mu(V)+c'\mu(X\setminus V).\]
                    We combine both inequalities to obtain
                        \[h_\mu(\sigma)\geq c-\int f\,d\mu\geq c-(c\mu(V)+c'\mu(X\setminus V))=(c-c')\mu(X\setminus V).\]
                    A rearrangement of the extremal terms yields the desired result.
                \end{proof}

                The lemma above illustrates the following fact: if we choose a potential $f$ for which the difference $c-c'$ large enough, an equilibrium measure $\mu$ (if it exists) assigns small measure to the set $X\setminus V$. We can choose $K$ as a periodic orbit associated to a given periodic measure and construct a family of continuous potentials with an increasing difference $c-c'$ and expect any accumulation point in $\mathcal{M}_\sigma(\Sigma)$ of the sequence of equilibrium measures to be a periodic measure. This is the strategy we follow in the next lemma.
                
                In order to use the classical thermodynamic formalism for the mixing subshift $\Sigma$, we need to consider Hölder potentials. This guarantees the finiteness of the topological pressure. 

                \begin{lemma}\label{lem:stronglymixingdense}
                    Let $(\Sigma,\sigma)$ be a mixing SFT. Every $\sigma$-invariant periodic measure is an accumulation point of a sequence of strongly mixing measures in $\mathcal{M}_\sigma(\Sigma)$.
                \end{lemma}

                \begin{proof}
                    Let $\mu\in\mathcal{M}_\sigma(\Sigma)$ be a fixed periodic measure, and let $K$ be the periodic orbit supporting it. Consider the sequence of neighborhoods of $K$ defined by $V_n\coloneq B(K,2^{-n})=\overline{B(K,2^{-(n+1)})}$. The compactness of $K$ and $\Sigma\setminus V_n$ implies that $d(K,\Sigma\setminus V_n)\eqcolon \delta_n>0$ for every $n\in\N$.
                    We define, for each $n\in\N$, the Hölder potential $f_n\colon X\to\mathbb{R}$ by
                        \[f_n(x)\coloneq n \max\{1-\delta_n^{-1}d(x,K),0\},\]
                    which is identically $n$ on $K$, vanishes in $\Sigma\setminus V_n$ and satisfies $0\leq f_n(x)\leq n$ for every $x\in X$. It is easy to verify that 
                        \[\sup_{k\in\N}\dfrac{\text{var}_k(f_n)}{2^{-k}}\leq n\delta_n^{-1}<\infty,\]
                    and consequently there exists a Gibbs measure $\mu_n$ associated to the Hölder potential $f_n$, as proved in Theorem 1.22 in \cite{Bowen}. The Gibbs measure $\mu_n$ is strongly mixing. 

                    By Lemma \ref{lem:urysohn}, with $c=n$ and $c'=0$, and the variational principle of the entropy
                        \[\mu_n(\Sigma\setminus V_n)\leq \dfrac{h_{\mu_n}(\sigma)}{n}\leq \dfrac{h(\sigma,\Sigma)}{n}.\]
                    Extract a convergent subsequence $\{\mu_{n_k}\}_{k\in\N}$ with limit $\nu\in\mathcal{M}(\Sigma)$. The support of $\nu$ is contained in $K$ since for every $m\in\N$
                        \[\nu(\Sigma\setminus V_m)\leq \liminf_{n_k\geq m} \mu_{n_k} (\Sigma\setminus V_m)\leq \liminf_{k\to\infty}\mu_{n_k}(\Sigma\setminus V_{n_k})\leq \liminf_{k\to\infty}\dfrac{h(\sigma,\Sigma)}{n_k}=0.\] 
                    The $\sigma$-invariance of $\nu$ implies that $\supp(\nu)=K$. In conclusion, $\mu=\nu$ and the claim is proved.                    
                \end{proof}

                Note that a direct consequence of this result is that strongly mixing measures are dense in $\mathcal{M}_\sigma(\Sigma)$, and so are the weakly mixing ones. We can now combine Lemma \ref{lem:weakmixingSFTGdelta} and Lemma \ref{lem:stronglymixingdense} to obtain the following result.

                \begin{theorem}
                    Let $(\Sigma,\sigma)$ be a mixing SFT over a finite alphabet and $\mathcal{M}_\sigma(\Sigma)$ the space of $\sigma$-invariant Borel probability measures. Then the weakly mixing measures are generic in $\mathcal{M}_\sigma(\Sigma)$. 
                \end{theorem}

            \subsection{Weakly mixing measures of subsystems in \texorpdfstring{$\overline{TT'(X,f)}$}{TT'(X,f)}}\label{subsec:WMsubsystem}
                In this section we obtain a genericity result concerning weakly mixing measures of subsystems in $\overline{TT'(X,f)}$. In the space of subshifts, studied in \cite{PavlovSchmieding}, the set of uniquely ergodic systems whose unique invariant measure is weakly mixing is a residual set in $\overline{TT'(\Sigma,\sigma)}$ whenever $\Sigma$ is a full shift. Since the current setting requires $(\Sigma,\sigma)$ to be a mixing SFT, the previous section extends this result to a broader class of subshifts.

                We start by translating genericity of weakly mixing measures in $\mathcal{M}_f(X)$ to genericity of subsystems in $\overline{TT'(X,f)}$ that are uniquely ergodic.
                
                \begin{lemma}
                    Let $(X,f)$ be a mixing Smale space and $\mathcal{M}_f(X)$ the space of $f$-invariant Borel probability measures on $X$ endowed with the weak-\textasteriskcentered\ topology. The function $g\colon \mathrm{UE}(X,f)\to \mathcal{M}_f(X)$ defined by $Y\mapsto \mu_Y$, where $\mu_Y$ is the unique invariant measure of $(Y,f)$, is continuous, i.e., if
                        \begin{enumerate}[font=\upshape]
                            \item each $(Y_n,f)$ is uniquely ergodic with $f$-invariant probability measure $\mu_n$,
                            \item $(Y,f)$ is uniquely ergodic with unique $f$-invariant probability measure $\mu_Y$,
                            \item and $Y_n\to Y$ in the Hausdorff metric,
                        \end{enumerate}
                    then $\mu_n\to\mu_Y$ in the weak-\textasteriskcentered\ topology.
                \end{lemma}

                \begin{proof}
                    Let $\{\mu_{n_j}\}_{j\in\N}$ be any subsequence of $\{\mu_n\}_{n\in\N}$ converging in the weak-\textasteriskcentered\ topology to a probability measure $\nu\in\mathcal{M}_f(X)$. It suffices to show that $\mu_Y=\nu$.
                    
                    Observe that $\nu$ is $f$-invariant. Also, $\supp(\nu)\subseteq Y$: let $x\in X\setminus Y$. Since $Y$ is closed, the distance from $x$ to $Y$ is strictly positive, namely $r\coloneq d(x,Y)>0$. Define $\eps=r/2$. We prove that $\nu(B(x,r/2))=0$, so $x\not\in\supp(\nu)$ follows. 
                    By convergence $Y_n\to Y$ in the Hausdorff metric, there exists $N\in\N$ such that $\forall n\geq N$
                        \[Y_n\subseteq \{z\in X \mid d(z,Y)<\eps\}.\]
                    On the other hand, if $z\in B(x,\eps)$, then $d(z,Y)\geq \eps$. Hence $B(x,\eps)\cap\{z\in X \mid d(z,Y)<\eps\}=\emptyset$. 
                    Therefore, for every $n\geq N$
                        \[Y_n\cap B(x,\eps)=\emptyset.\]
                    Each measure $\mu_n$ is supported on $Y_n$, which yields $\mu_n(B(x,\eps))=0$ for every $n\geq N$. Passing to the  convergent subsequence $\mu_{n_j}$ we obtain
                        \[\nu(B(x,\eps))\leq\liminf \mu_{n_j}(B(x,\eps))=0.\]
                    It follows that $\nu$ is $f$-invariant and supported on $Y$, but $(Y,f)$ is uniquely ergodic with unique invariant measure $\mu_Y$, so it must be $\nu=\mu_Y$.    
                \end{proof}

                The previous lemma allows us to conclude that if $\text{WMM}(X,f)$ is a $G_\delta$ set of $\mathcal{M}_f(X)$, then $g^{-1}(\text{WMM}(X,f))$ is a $G_\delta$ subset of $\text{UE}(X,f)$ and therefore a $G_\delta$ subset of $\overline{TT'(X,f)}$. 

                Define
                    \[\text{WM}(X,f)\coloneq \{Y\in \text{UE}(X,f) \mid \mu_Y \in \mathcal{M}_f(X) \text{ is weakly mixing}\},\]
                so $g^{-1}(\text{WMM}(X,f))=\text{WM}(X,f)$.

                \begin{theorem}\label{thm:ueweakmixinggeneric}
                    Let $(X,f)$ be a mixing Smale space. The family of uniquely ergodic subsystems with a weakly mixing invariant measure $\text{WM}(X,f)$ is generic in $\overline{TT'(X,f)}$.
                \end{theorem}

                \begin{proof}
                    We already know that $\text{WM}(X,f)$ is a $G_\delta$ subset of $\overline{TT'(X,f)}$, so it remains to show its density. Note that  $\pi:(\Sigma,\sigma)\to(X,f)$ is a continuous factor, so $\pi_*:S(\Sigma,\sigma)\to S(X,f)$ preserves unique ergodicity. It is also a measurable factor, so the push-forward measure $\pi_*\mu\in\mathcal{M}_f(X)$ of a weakly mixing measure $\mu\in\mathcal{M}_\sigma(\Sigma)$ is weakly mixing. Therefore, 
                        \[\pi_*(\text{WM}(\Sigma,\sigma)\cap \overline{TT'(\Sigma,\sigma)})\subseteq \text{WM}(X,f)\cap \overline{TT'(X,f)}\] 
                    is a dense subset of $\overline{TT'(X,f)}$ by Theorem 6.8 in \cite{PavlovSchmieding}. 
                \end{proof}

                Note that if $(Y,f)$ is a minimal uniquely ergodic subsystem with a weakly mixing measure $\mu_Y$, then $(Y\times Y,f\times f,\mu\times\mu)$ is ergodic and $\supp(\mu\times\mu)=Y\times Y$. This implies that $(Y,f)$ is topologically weakly mixing.  

                \begin{corollary}
                    Let $(X,f)$ be a mixing Smale space. The set of topologically weakly mixing subsystems in $S(X,f)$ is residual in $\overline{TT'(X,f)}$.
                \end{corollary}

                Observe that Theorem \ref{thm:ueweakmixinggeneric} shows that the dense conjugacy class of a Sturmian subshift in Proposition \ref{prop:sturmiandense} is meager, since a Sturmian subshift never has a weakly mixing invariant measure, and neither does any element of its conjugacy class. Alternatively, such a conjugacy class is countable and $\overline{TT'(X,f)}$ is perfect, so this conjugacy class is always meager.

                \section*{Acknowledgment of AI use}
                The author used large language models for syntax and grammar corrections and reference searches.
  
    \bibliography{Genericity.bib}
\end{document}